\documentclass{amsart}
\usepackage{setspace}
\usepackage{a4}
\usepackage{amssymb,amsmath,amsthm,latexsym}
\usepackage{amsfonts}
\usepackage{amsfonts}
\usepackage{graphicx}
\usepackage{textcomp}
\usepackage{cite}
\usepackage{enumerate}
\usepackage[mathscr]{euscript}
\usepackage{mathtools}
\newtheorem{theorem}{Theorem}[section]

\newtheorem{corollary}[theorem] {Corollary}
\newtheorem{definition}[theorem]{Definition}
\newtheorem{example}[theorem]{Example}
\newtheorem{lemma} [theorem]{Lemma}

\newtheorem{proposition}[theorem]{Proposition}

\title{This is the title}
\usepackage{amssymb}
\usepackage{amssymb}
\usepackage{amssymb}
\usepackage{amssymb}
\usepackage{amsmath}
\usepackage{tikz}
\usepackage{hyperref}
\usepackage{enumerate}
\usepackage{mathtools}
\usepackage{amsmath}
\usepackage{tikz}
\usepackage{amssymb}
\usepackage{amsmath}
\usepackage{tikz}
\usepackage{hyperref}
\usepackage{tikz}
\usepackage{graphicx}
\usepackage{textcomp}
\usetikzlibrary{cd}
\begin{document}
\begin{center}
{\bf{FACTORABLE WEAK OPERATOR-VALUED FRAMES}}\\
K. MAHESH KRISHNA AND P. SAM JOHNSON  \\
Department of Mathematical and Computational Sciences\\ 
National Institute of Technology Karnataka (NITK), Surathkal\\
Mangaluru 575 025, India  \\
Emails: kmaheshak@gmail.com,  sam@nitk.edu.in\\

Date: \today
\end{center}

\hrule
\vspace{0.5cm}
\textbf{Abstract}: Let $\mathcal{H}$ and $\mathcal{H}_0$ be Hilbert spaces and $\{A_n\}_n$ be a sequence of bounded linear operators from $\mathcal{H}$ to $\mathcal{H}_0$. The study frames for Hilbert spaces initiated the study of operators of the form $\sum_{n=1}^{\infty}A_n^*A_n$, where the convergence is in the strong-operator topology, by Kaftal, Larson and Zhang in the paper: Operator-valued frames. \textit{Trans. Amer. Math. Soc.}, 361(12):6349-6385, 2009.  In this paper, we generalize this and  study the  series of the form $\sum_{n=1}^{\infty}\Psi_n^*A_n$, where $\{\Psi _n\}_n$ is a sequence of operators from $\mathcal{H}$ to $\mathcal{H}_0$. Main tool used in the study of $\sum_{n=1}^{\infty}A_n^*A_n$ is the factorization of this series. Since the series $\sum_{n=1}^{\infty}\Psi_n^*A_n$ may not be factored, it demands greater care. Therefore we impose a factorization of $\sum_{n=1}^{\infty}\Psi_n^*A_n$ and derive various results. We characterize them and derive dilation results. We further study the series by taking the indexed set as group as well as group-like unitary system. We also derive stability results.




\section{Introduction}
Harmonic Fourier series deals with the representation of functions in the Hilbert space $\mathcal{L}^2[-\pi, \pi]$ using the sequence $ \{e^{inx}\}_{n} $ whereas non-harmonic Fourier series deals with representation of functions  in the Banach space $\mathcal{L}^p[-\pi, \pi]$ using the sequence $ \{e^{i\lambda_nx}\}_{n}$ \cite{YOUNG}.  It was in 1952,  in the study of non-harmonic Fourier series,    Duffin and Schaeffer introduced frames for Hilbert spaces \cite{DUFFIN}.  For a span of three decades,  frames are not studied further. Work of  Daubechies, Grossmann, and Meyer \cite{MEYER1} gave rebirth to the frames in 1986. Today, frame theory stands on its own feet and the references \cite{CHRISTENSENBOOK, HANLARSON, ALDROUBI,  HEILBOOK}  will give a comprehensive look on frames (for infinite dimensional separable Hilbert spaces). Formal definition of frames for Hilbert spaces reads as follows. 
\begin{definition}\cite{DUFFIN}\label{OLE}
	A collection $ \{\tau_n\}_{n}$ in  a separable Hilbert space $ \mathcal{H}$ is said to be a frame for $\mathcal{H}$ if there exist $ a, b >0$ such that
	\begin{equation}\label{SEQUENTIALEQUATION1}
		a\|h\|^2\leq\sum_{n=1}^\infty|\langle h, \tau_n \rangle|^2 \leq b\|h\|^2  ,\quad \forall h \in \mathcal{H}.
	\end{equation}
	Constants $ a$ and $ b$ are called as lower and upper frame bounds, respectively. If $a=1=b$, then the frame is   called  as Parseval frame.
\end{definition}
Historically, many generalizations of frames for Hilbert spaces are proposed such as frames for subspaces \cite{CASAZZASUBSPACE}, fusion frames \cite{CASAZZAFUSION}, outer frames \cite{OUTER}, oblique frames \cite{CHRISTENSENOBLIQUE}, pseudo frames \cite{LIPSEUDO}, quasi-projectors \cite{FORNASIERQUASI}. It was in 2006, Sun gave the definition of G-frame which unified all these notions of frames for Hilbert spaces \cite{SUN1}. We use following notations in this paper. Letter $ \mathcal{H}$ always denotes a Hilbert space, so is any of its `integer' subscripts. Identity operator on $ \mathcal{H}$ is denoted by $ I_\mathcal{H}$. Banach space of all bounded linear operators from $ \mathcal{H}$ to $ \mathcal{H}_0 $ is denoted by $ \mathcal{B}(\mathcal{H}, \mathcal{H}_0)$. We write $ \mathcal{B}(\mathcal{H}, \mathcal{H})$ as  $ \mathcal{B}(\mathcal{H})$.
\begin{definition}\cite{SUN1}\label{SUNDEF}
	A collection  $ \{A_n\}_{n} $  in $ \mathcal{B}(\mathcal{H}, \mathcal{H}_0)$ is said to be a G-frame in  $ \mathcal{B}(\mathcal{H}, \mathcal{H}_0)$ if there exist $ a, b >0$ such that 
	\begin{align*}
	a\|h\|^2\leq\sum_{n=1}^\infty\|A_nh\|^2 \leq b\|h\|^2  ,\quad \forall h \in \mathcal{H}.	
	\end{align*}
	\end{definition}
Through a decade long research (see the introduction in \cite{KAFTAL}),  Kaftal, Larson, and Zhang defined the notion of operator-valued frames (OVFs) which is equivalent to the notion of G-frames. Basic idea for the notion of OVF is the following.   Definition \ref{OLE} can be written in an equivalent form as 
\begin{equation}\label{SEQUENTIALEQUATION2}
\text{the map}~ S_\tau: \mathcal{H} \ni h \mapsto \sum_{n=1}^\infty\langle h, \tau_n\rangle \tau_n \in \mathcal{H} ~\text{is well-defined bounded positive invertible operator}.
\end{equation}
If we now define $ A_n : \mathcal{H} \ni  h \mapsto \langle h, x_n\rangle \in \mathbb{K} $, for each $ n \in \mathbb{N}$, then  one more way for Statement (\ref{SEQUENTIALEQUATION2}) is 
\begin{equation}\label{SEQUENTIALEQUATION3}
\sum_{n=1}^\infty A_n^*A_n ~\text{converges in the strong-operator topology  on } \mathcal{B}(\mathcal{H}) \text{ to a bounded positive invertible operator.}
\end{equation}
Now Statement (\ref{SEQUENTIALEQUATION3}) leads to 

\begin{definition}\cite{KAFTAL}\label{KAFTAL}
	A collection  $ \{A_n\}_{n} $  in $ \mathcal{B}(\mathcal{H}, \mathcal{H}_0)$ is said to be an OVF in $ \mathcal{B}(\mathcal{H}, \mathcal{H}_0)$ if the series 
	\begin{align*}
\text{(frame operator)}\quad 	S_A\coloneqq \sum_{n=1}^\infty A_n^*A_n
	\end{align*}
	 converges in the strong-operator topology on $ \mathcal{B}(\mathcal{H})$ to a  bounded invertible operator.
\end{definition}
The fundamental tool used in the study of OVF is the factorization of frame operator $S_A$. This and other important properties of OVFs are stated in the following theorem. To state the theorem we need a particular collection of operators described as follows.  Following \cite{KAFTAL}, given $n \in \mathbb{N}$, we define 
\begin{align*}
L_n : \mathcal{H}_0 \ni h \mapsto L_nh\coloneqq e_n\otimes h \in  \ell^2(\mathbb{N}) \otimes \mathcal{H}_0,
\end{align*}  where $\{e_n\}_{n} $ is  the standard orthonormal basis for $\ell^2(\mathbb{N})$. It then follows that $L_n$'s are isometries from  $\mathcal{H}_0 $ to $ \ell^2(\mathbb{N}) \otimes \mathcal{H}_0$, and  for
$  n,m \in \mathbb{N}$ we have  
\begin{align}\label{LEQUATION}
L_n^*L_m =
\left\{
\begin{array}{ll}
I_{\mathcal{H}_0 } & \mbox{if } n=m \\
0 & \mbox{if } n\neq m
\end{array}
\right.
~\text{and} \quad 
\sum\limits_{n=1}^\infty L_nL_n^*=I_{\ell^2(\mathbb{N})}\otimes I_{\mathcal{H}_0}
\end{align}
where  the convergence is in the strong-operator topology. We also have 
$L_m^*(\{a_n\}_{n}\otimes y) =a_my, \forall  \{a_n\}_{n} \in \ell^2(\mathbb{N}), \forall y \in \mathcal{H}_0$, for each $ m $ in $ \mathbb{N}.$ 
\begin{theorem}\cite{KAFTAL}
	Let  $ \{A_n\}_{n} $  be an OVF in $ \mathcal{B}(\mathcal{H}, \mathcal{H}_0)$. Then
\begin{enumerate}[\upshape(i)]
	\item The analysis operator 
	\begin{align*}
 \theta_A:\mathcal{H} \ni h \mapsto   \theta_A h\coloneqq\sum_{n=1}^\infty L_nA_n h \in \ell^2(\mathbb{N}) \otimes \mathcal{H}_0
\end{align*}
is a well-defined bounded  linear injective operator.
\item The synthesis operator 
\begin{align*}
\theta_A^*:\ell^2(\mathbb{N})\otimes \mathcal{H}_0 \ni z\mapsto\sum\limits_{n=1}^\infty A_n^*L_n^*z \in \mathcal{H} 
\end{align*}
is a well-defined bounded  linear surjective operator.
\item Frame operator 
	factors  as $S_A=\theta_A^*\theta_A.$
\item $ P_A \coloneqq \theta_A S_A^{-1} \theta_A^*:\ell^2(\mathbb{N})\otimes \mathcal{H}_0 \to \ell^2(\mathbb{N})\otimes \mathcal{H}_0$ is an orthogonal  projection onto $ \theta_A(\mathcal{H})$.
\end{enumerate}	
\end{theorem}
This paper is organized as follows. In Section \ref{MK} we define the notion of weak OVF which is about the invertibility of the series $\sum_{n=1}^{\infty}\Psi_n^*A_n$. We then define  the notion of dual frames. To study further properties of weak OVFs, we define factorable weak OVFs which  allows to decompose the series as product of two bounded linear operators. This notion also allows to move frequently between Hilbert spaces $\mathcal{H}$ and $\ell^2(\mathbb{N})\otimes \mathcal{H}_0$. We then derive dilation result and various characterizations for weak OVFs. We also derive characterizations for  duals of weak OVFs. Orthogonality of weak OVFs are introduced and an interpolation result is derived. Section \ref{SIMILARITYCOMPOSITIONANDTENSORPRODUCT} contains the notion of equivalence of weak OVFs. Equivalence depends upon natural numbers. We  characterize them using operators. Section \ref{FRAMESANDDISCRETEGROUPREPRESENTATIONS} studies factorable weak OVFs indexed by groups. Advantage of indexing a frame with group is that  we can generate factorable weak OVFs by starting with two fixed operators. Theorem \ref{gc1} characterizes unitary representations of discrete groups which generate factorable Parseval weak OVFs. In Section \ref{FRAMESANDGROUP-LIKEUNITARYSYSTEMS} we study OVFs indexed by group-like unitary systems.  Theorem  \ref{CHARACTERIZATIONGROUPLIKE} characterizes  unitary representations of group-like unitary system   which generate Parseval weak OVFs. Section \ref{PERTURBATIONS} shows that factorable weak OVFs are  stable under perturbations.

\section{Factorable weak operator-valued frames}\label{MK}

\begin{definition}
	Let  $ \{A_n\}_{n} $ and  $ \{\Psi_n\}_{n} $ be collections in $ \mathcal{B}(\mathcal{H}, \mathcal{H}_0)$. The pair $( \{A_n\}_{n},  \{\Psi_n\}_{n} )$ is said to be a weak \textit{operator-valued frame}  (weak OVF) in $ \mathcal{B}(\mathcal{H}, \mathcal{H}_0) $   if  the series 
	\begin{align*}
	\text{(frame operator)}\quad S_{A, \Psi} \coloneqq  \sum_{n=1}^\infty \Psi_n^*A_n
	\end{align*}
	  converges in the strong-operator topology on $ \mathcal{B}(\mathcal{H})$ to a  bounded  invertible operator. If $	S_{A, \Psi}=I_\mathcal{H}$, then the frame is called as a Parseval frame. 
\end{definition}
 Unlike in the case of OVFs, note that the operator $S_{A, \Psi}$ need not be positive.  Since $S_{A, \Psi}$ is invertible, there are $a,b>0$ such that 
\begin{align*}
a\|h\|\leq \|S_{A, \Psi}h\|\leq b \|h\|, \quad \forall h \in \mathcal{H}.
\end{align*}
We call such $a,b$ as lower and upper frame bounds, respectively. Supremum of the set of all lower frame bounds is called as optimal lower frame bound and infimum of the set of all upper frame bounds is called as optimal upper frame bound. We easily get that 
\begin{align*}
\text{ optimal lower frame bound }=\|S_{A,\Psi}^{-1}\|^{-1} \quad \text{and} \quad  \text{ optimal upper frame bound } = \|S_{A,\Psi}\|.
\end{align*}
Given a frame $ \{\tau_n\}_{n}$ for  a separable Hilbert space $ \mathcal{H}$, it is known in frame theory that $ \{S_\tau^{-1}\tau_n\}_{n}$ is again a frame for $ \mathcal{H}$. This frame is known as dual frame. We now define such a notion for weak OVFs.
\begin{definition}
	A weak  OVF  $ (\{B_n\}_{n} , \{\Phi_n\}_{n} )$  in $\mathcal{B}(\mathcal{H}, \mathcal{H}_0)$ is said to be dual for a weak  OVF $  ( \{A_n\}_{n},  \{\Psi_n\}_{n} )$ in $\mathcal{B}(\mathcal{H}, \mathcal{H}_0)$  if  
	\begin{align*}
	\sum_{n=1}^\infty \Psi_n^*B_n= \sum_{n=1}^\infty\Phi^*_nA_n=I_{\mathcal{H}}.
	\end{align*}
\end{definition}
	Note that  dual always exists for a given weak OVF.  In fact, a direct calculation shows that, each of 
	\begin{align*}
	  ( \{\widetilde{A}_n\coloneqq A_nS_{A,\Psi}^{-1}\}_{n},\{\widetilde{\Psi}_n\coloneqq\Psi_n(S_{A,\Psi}^{-1})^*\}_{n})
	\end{align*}
	is a weak OVF and is a dual for $  ( \{A_n\}_{n},  \{\Psi_n\}_{n} )$.  This weak  OVF  is as called the canonical dual for $  ( \{A_n\}_{n},  \{\Psi_n\}_{n} )$. Canonical duals have two nice properties. Following two results establish them.

\begin{proposition}
	Let $( \{A_n\}_{n},  \{\Psi_n\}_{n} )$ be a weak  OVF in $ \mathcal{B}(\mathcal{H}, \mathcal{H}_0).$  If $ h \in \mathcal{H}$ has representation  $ h=\sum_{n=1}^\infty A_n^*y_n= \sum_{n=1}^\infty\Psi_n^*z_n, $ for some sequences $ \{y_n\}_{n},\{z_n\}_{n}$ in $ \mathcal{H}_0$, then 
	$$ \sum_{n=1}^\infty\langle y_n,z_n\rangle =\sum_{n=1}^\infty\langle \widetilde{\Psi}_nh,\widetilde{A}_nh\rangle+\sum_{n=1}^\infty\langle y_n-\widetilde{\Psi}_nh,z_n-\widetilde{A}_nh\rangle. $$
\end{proposition}  
\begin{proof}
We start from the right side and see 	
	\begin{align*}
&\sum\limits_{n=1}^\infty\langle \widetilde{\Psi}_nh,\widetilde{A}_nh\rangle+\sum\limits_{n=1}^\infty\langle y_n, z_n\rangle -\sum\limits_{n=1}^\infty\langle y_n, \widetilde{A}_nh\rangle-\sum\limits_{n=1}^\infty\langle \widetilde{\Psi}_nh, z_n\rangle +\sum\limits_{n=1}^\infty\langle \widetilde{\Psi}_nh,\widetilde{A}_nh\rangle\\
&=2\sum\limits_{n=1}^\infty\langle \widetilde{\Psi}_nh,\widetilde{A}_nh\rangle+ \sum\limits_{n=1}^\infty\langle y_n, z_n\rangle-\sum\limits_{n=1}^\infty\langle y_n,A_nS_{A,\Psi}^{-1}h\rangle-\sum\limits_{n=1}^\infty\langle \Psi_n(S_{A,\Psi}^{-1})^*h, z_n\rangle\\
&= 2\left\langle\sum\limits_{n=1}^\infty(S_{A,\Psi}^{-1})^*A_n^*\Psi_n(S_{A,\Psi}^{-1})^*h, h \right\rangle+ \sum\limits_{n=1}^\infty\langle y_n, z_n\rangle-\left\langle \sum\limits_{n=1}^\infty A_n^*y_n,S_{A,\Psi}^{-1}h\right \rangle -\left\langle (S_{A,\Psi}^{-1})^*h , \sum\limits_{n=1}^\infty\Psi_n^*z_n\right \rangle\\
&=2 \langle (S_{A,\Psi}^{-1})^*h,h \rangle + \sum\limits_{n=1}^\infty\langle y_n, z_n\rangle -\langle h, S_{A,\Psi}^{-1}h\rangle-\langle (S_{A,\Psi}^{-1})^*h, h\rangle
\end{align*}
which is the left side.	
\end{proof}
\begin{theorem}\label{CANONICALDUALFRAMEPROPERTYOPERATORVERSIONWEAK}
	Let $( \{A_n\}_{n},  \{\Psi_n\}_{n} )$ be a weak  OVF with frame bounds $ a$ and $ b.$ Then
	\begin{enumerate}[\upshape(i)]
		\item The canonical dual weak  OVF for the canonical dual weak  OVF  for $( \{A_n\}_{n},  \{\Psi_n\}_{n} )$ is itself.
		\item$ \frac{1}{b}, \frac{1}{a}$ are frame bounds for the canonical dual of $( \{A_n\}_{n},  \{\Psi_n\}_{n} )$.
		\item If $ a, b $ are optimal frame bounds for $( \{A_n\}_{n},  \{\Psi_n\}_{n} )$, then $ \frac{1}{b}, \frac{1}{a}$ are optimal  frame bounds for its canonical dual.
	\end{enumerate} 
\end{theorem} 
\begin{proof}
Since (ii) and (iii)  follow from the property of invertible operators on Banach spaces, we have to argue for (i)$(\{A_nS_{A,\Psi}^{-1}\}_{n}, \{\Psi_n(S_{A,\Psi}^{-1})^*\}_{n} )$ is 
	$$ \sum\limits_{n=1}^\infty(\Psi_n(S_{A,\Psi}^{-1}))^* (A_nS_{A,\Psi}^{-1}) =S_{A,\Psi}^{-1}\left(\sum\limits_{n=1}^\infty\Psi_n ^*A_n\right)S_{A,\Psi}^{-1} =S_{A,\Psi}^{-1}S_{A,\Psi}S_{A,\Psi}^{-1}= S_{A,\Psi}^{-1}.$$
	Therefore, its canonical dual is $(\{(A_nS_{A,\Psi}^{-1})S_{A,\Psi}\}_{n} , \{(\Psi_nS_{A,\Psi}^{-1})S_{A,\Psi}\}_{n})$ which is the original frame.
\end{proof}
For the further study of weak OVFs, we impose some conditions so that the frame operator splits.  
\begin{definition}
	A weak OVF $( \{A_n\}_{n},  \{\Psi_n\}_{n} )$ is said to be factorable if both the maps 
	\begin{align*}
\text{ (analysis operator) }	\quad  \theta_A:\mathcal{H} \ni h \mapsto   \theta_A h\coloneqq\sum_{n=1}^\infty L_nA_n h \in \ell^2(\mathbb{N}) \otimes \mathcal{H}_0\\
\text{ (analysis operator) }	\quad	 \theta_\Psi:\mathcal{H}\ni h \mapsto  \theta_\Psi h\coloneqq \sum_{n=1}^\infty L_n\Psi_n h \in \ell^2(\mathbb{N}) \otimes \mathcal{H}_0
	\end{align*}
	are well-defined bounded linear operators. 
\end{definition}
We next give an example which shows that a weak OVF need not be factorable.

\begin{example}
	On $ \mathbb{C},$ define $ A_nx\coloneqq\frac{x}{\sqrt{n}},  \forall x \in \mathbb{C}, \forall n \in \mathbb{N}$, and $\Psi_1x\coloneqq x, \Psi_nx\coloneqq0, \forall x \in \mathbb{C}, \forall n \in \mathbb{N}\setminus\{1\} $. Then $ \sum_{n=1}^\infty\Psi_n^*A_nx$ converges to a positive invertible operator but  $ \sum_{n=1}^\infty L_nA_nx$ doesn't. In fact, using Equation (\ref{LEQUATION}),
	\begin{align*}
	\left\| \sum_{n=1}^mL_nA_n1\right\|^2=\sum_{n=1}^m\|A_n1\|^2=\sum_{n=1}^m\frac{1}{n} \to \infty \quad \text{ as } \quad m \to \infty.
	\end{align*}
\end{example}
 Equation (\ref{LEQUATION}) gives the following theorem  easily.
\begin{theorem}
	Let  $ \{A_n\}_{n} $  be a weak  OVF in $ \mathcal{B}(\mathcal{H}, \mathcal{H}_0)$. Then
	\begin{enumerate}[\upshape(i)]
		\item The analysis operator 
		\begin{align*}
		\theta_A:\mathcal{H} \ni h \mapsto   \theta_A h\coloneqq\sum_{n=1}^\infty L_nA_n h \in \ell^2(\mathbb{N}) \otimes \mathcal{H}_0
		\end{align*}
		is a well-defined bounded  linear injective operator.
		\item The synthesis operator 
		\begin{align*}
		\theta_\Psi^*:\ell^2(\mathbb{N})\otimes \mathcal{H}_0 \ni z\mapsto\sum\limits_{n=1}^\infty \Psi_n^*L_n^*z \in \mathcal{H} 
		\end{align*}
		is a well-defined bounded  linear surjective operator.
		\item Frame operator 
		factors  as $S_{A,\Psi}=\theta_\Psi^*\theta_A.$
		\item $ P_{A,\Psi} \coloneqq \theta_A S_{A,\Psi}^{-1} \theta_\Psi^*:\ell^2(\mathbb{N})\otimes \mathcal{H}_0 \to \ell^2(\mathbb{N})\otimes \mathcal{H}_0$ is an  idempotent onto $ \theta_A(\mathcal{H})$.
	\end{enumerate}	
\end{theorem}
We next define the notions of Riesz and orthonormal factorable weak OVFs. For OVFs, these notions were defined by Kaftal, Larson, and Zhang \cite{KAFTAL}.
\begin{definition}\label{RIESZOVF}
	A factorable weak OVF  $( \{A_n\}_{n},  \{\Psi_n\}_{n} )$  in $\mathcal{B}(\mathcal{H}, \mathcal{H}_0)$ is said to be a Riesz OVF   if $ P_{A,\Psi}= I_{\ell^2(\mathbb{N})}\otimes I_{\mathcal{H}_0}$. A Parseval and  Riesz OVF, i.e., $\theta_\Psi^*\theta_A=I_\mathcal{H} $ and  $\theta_A\theta_\Psi^*=I_{\ell^2(\mathbb{N})}\otimes I_{\mathcal{H}_0} $ is called as an orthonormal OVF. 
\end{definition}
\begin{proposition}\label{ORTHORESULT}
	A factorable weak OVF  $( \{A_n\}_{n},  \{\Psi_n\}_{n} )$ in $ \mathcal{B}(\mathcal{H}, \mathcal{H}_0)$ is an orthonormal OVF  if and only if it is a Parseval OVF and $ A_n\Psi_m^*=\delta_{n,m}I_{\mathcal{H}_0},\forall n,m \in \mathbb{N}$. 
\end{proposition}
\begin{proof}
	$(\Rightarrow)$	
	 We have $\theta_A\theta_\Psi^*=I_{\ell^2(\mathbb{N})}\otimes I_{\mathcal{H}_0}.$ Hence $ e_m\otimes y=\theta_A\theta_\Psi^*(e_m\otimes y)=\sum_{n=1}^\infty L_nA_n(\sum_{k=1}^\infty\Psi^*_kL_k^*(e_m\otimes y))=\sum_{n=1}^\infty L_nA_n\Psi^*_my=\sum_{n=1}^\infty(e_n\otimes A_n\Psi^*_my)= e_m\otimes( A_n\Psi^*_m y+\sum_{n=1, n\neq m}^\infty(e_n\otimes A_n\Psi^*_my)),\forall m \in \mathbb{N}, \forall y \in\mathcal{H}_0 $. We then have  $A_n\Psi^*_my=\delta_{n,m}y,\forall y \in \mathcal{H}_0$.
	
	$(\Leftarrow)$ $ \theta_A\theta_\Psi^*=\sum_{n=1}^\infty L_nA_n(\sum_{k=1}^\infty\Psi_k^*L_k^*)=\sum_{n=1}^\infty L_nL_n^*=I_{\ell^2(\mathbb{N})}\otimes I_{\mathcal{H}_0}.$
	
\end{proof}
 In the theory of frames for Hilbert spaces, following result is known as (Naimark) dilation theorem. It was obtained independently by Han and Larson  \cite{HANLARSON} and  Kashin and Kulikova \cite{KASHINKULIKOVA} (also see \cite{CZAJA}). 
 \begin{theorem}\cite{HANLARSON, KASHINKULIKOVA}\label{NAIMAR}
 	If $ \{x_n\}_{n}$ is   a Parseval frame for a Hilbert space $ \mathcal{H}$, then there exist a Hilbert space $ \mathcal{H}_1 $ which contains $ \mathcal{H}$ isometrically, an orthonormal basis $ \{y_n\}_{n}$ for  $ \mathcal{H}_1$ and an orthogonal projection $P: \mathcal{H}_1 \to  \mathcal{H} $ such that $Py_n=x_n$, $\forall n$.
 \end{theorem}
 
 Theorem \ref{NAIMAR} was generalized to OVFs in \cite{HANLIMENGTANG} which is known as  general Naimark-Han-Larson dilation theorem \cite{HANLIMENGTANG}. We now generalize this to factorable weak OVFs. First we need a lemma for this. 
 \begin{lemma}\label{DILATIONLEMMA}
 Let $( \{A_n\}_{n},  \{\Psi_n\}_{n} )$ be  a factorable  weak OVF  in $ \mathcal{B}(\mathcal{H}, \mathcal{H}_0)$. Then the range of 	$\theta_A $ is closed.
 \end{lemma}
\begin{proof}
Let $ \{h_n\}_{n=1}^\infty$ in $\mathcal{H} $  be such that  $ \{\theta _Ah_n\}_{n=1}^\infty$ converges to $ y \in \mathcal{H}_0$. This gives $ S_{A,\Psi}h_n \rightarrow \theta_\Psi^*y$  as $ n \rightarrow \infty$ and this in turn gives $h_n \rightarrow S_{A,\Psi}^{-1} \theta_\Psi^*y $  as $ n \rightarrow \infty.$ An application of $ \theta_A$ gives $\theta_Ah_n \rightarrow \theta_AS_{A,\Psi}^{-1} \theta_\Psi^*y $ as $ n \rightarrow \infty.$ Therefore $ y=\theta_A(S_{A,\Psi}^{-1} \theta_\Psi^*y).$	
\end{proof}
\begin{theorem}\label{OPERATORDILATION}
	Let $( \{A_n\}_{n},  \{\Psi_n\}_{n} )$ be  a Parseval weak OVF  in $ \mathcal{B}(\mathcal{H}, \mathcal{H}_0)$ such that $ \theta_A(\mathcal{H})=\theta_\Psi(\mathcal{H})$ and $ P_{A,\Psi}$ is projection. Then there exist a Hilbert space $ \mathcal{H}_1 $ which contains $ \mathcal{H}$ isometrically and  bounded linear operators $B_n,\Phi_n:\mathcal{H}_1\rightarrow \mathcal{H}_0, \forall n  $ such that $(\{B_n\}_{n} ,\{\Phi_n\}_{n})$ is an orthonormal OVF in $ \mathcal{B}(\mathcal{H}_1, \mathcal{H}_0)$ and $B_n|_{\mathcal{H}}=  A_n,\Phi_n|_{\mathcal{H}}=\Psi_n, \forall n \in \mathbb{N}$. 
\end{theorem}
\begin{proof}
	We first see  that  $P_{A,\Psi}$ is the  orthogonal projection from $ \ell^2(\mathbb{N})\otimes \mathcal{H}_0$ onto $\theta_A(\mathcal{H})=\theta_\Psi(\mathcal{H})$. Define $ \mathcal{H}_1\coloneqq\mathcal{H}\oplus \theta_A(\mathcal{H})^\perp$. From Lemma \ref{DILATIONLEMMA}, $\mathcal{H}_1$ becomes a Hilbert space. Then $\mathcal{H} \ni h \mapsto h\oplus 0 \in \mathcal{H}_1 $ is an isometry. Set $P_{A,\Psi}^\perp\coloneqq I_{\ell^2(\mathbb{N})\otimes \mathcal{H}_0}-P_{A,\Psi}$ and  define 
	\begin{align*}
	 B_n:\mathcal{H}_1\ni h\oplus g\mapsto A_nh+L_n^*P_{A,\Psi}^\perp g \in \mathcal{H}_0,  \quad \Phi_n:\mathcal{H}_1\ni h\oplus g\mapsto \Psi_nh+L_n^*P_{A,\Psi}^\perp g \in \mathcal{H}_0 , \quad \forall n \in \mathbb{N}.
	\end{align*}
	Then  clearly $B_n|_{\mathcal{H}}=  A_n,\Phi_n|_{\mathcal{H}}=\Psi_n, \forall n \in \mathbb{N}$. Now 
	\begin{align*}
	\theta_B(h\oplus g)=\sum_{n=1}^\infty L_nA_nh+\sum_{n=1}^\infty L_nL_n^*P_{A,\Psi}^\perp g=\theta_Ah+P_{A,\Psi}^\perp g, \quad \forall  h\oplus g \in \mathcal{H}_1.
	\end{align*}
	 Similarly $\theta_\Phi(h\oplus g)=\theta_\Psi h+P_{A,\Psi}^\perp g, \forall h\oplus g\in \mathcal{H}_1 $.  Also 
	 \begin{align*}
	 \langle \theta_B^*z,h\oplus g \rangle&= \langle z,  \theta_B(h\oplus g) \rangle = \langle \theta_A^*z,  h\rangle+\langle P_{A,\Psi}^\perp z,  g\rangle \\
	 &= \langle \theta_A^*z\oplus P_{A,\Psi}^\perp z, h\oplus g\rangle , \quad \forall z \in \ell^2(\mathbb{N})\otimes \mathcal{H}_0 , \forall  h\oplus g \in \mathcal{H}_1.
	 \end{align*}
	  Hence $\theta_B^*z=\theta_A^*z\oplus P_{A,\Psi}^\perp z, \forall z \in \ell^2(\mathbb{N})\otimes \mathcal{H}_0 $ and similarly $\theta_\Phi^*z=\theta_\Psi^*z\oplus P_{A,\Psi}^\perp z, \forall z \in \ell^2(\mathbb{N})\otimes \mathcal{H}_0$.
	By  using $\theta_A(\mathcal{H})=\theta_\Psi(\mathcal{H}) $ and $\theta_\Psi^*P_{A,\Psi}^\perp=0=P_{A,\Psi}^\perp\theta_A ,$ we get  \begin{align*}
	S_{B,\Phi}(h\oplus g)&= \theta_\Phi^*(\theta_Ah+ P_{A,\Psi}^\perp g)=\theta_\Psi^*(\theta_Ah+P_{A,\Psi}^\perp g)\oplus P_{A,\Psi}^\perp(\theta_Ah+P_{A,\Psi}^\perp g)\\
	&=(S_{A,\Psi}h+0)\oplus(0+P_{A,\Psi}^\perp g)=S_{A,\Psi}h\oplus P_{A,\Psi}^\perp g\\
	&=I_\mathcal{H}h\oplus I_{\theta_A(\mathcal{H})^\perp}g, \quad \forall h\oplus g\in \mathcal{H}_1.
	\end{align*}
	 Hence $(\{B_n\}_{n},\{\Phi_n\}_{n} )$ is  a Parseval weak OVF in $ \mathcal{B}(\mathcal{H}_1, \mathcal{H}_0)$. We further find 
	 \begin{align*}
	 P_{B,\Phi}z&=\theta_BS_{B,\Phi}^{-1}\theta_\Phi^*z=\theta_B\theta_\Phi^*z=\theta_B(\theta_\Psi^*z\oplus P_{A,\Psi}^\perp z)=\theta_A(\theta_\Psi^*z)+ P_{A,\Psi}^\perp(P_{A,\Psi}^\perp z)=P_{A,\Psi} z+P_{A,\Psi}^\perp z\\
	 &=P_{A,\Psi} z+((I_{\ell^2(\mathbb{N})}\otimes I_{\mathcal{H}_0})-P_{A,\Psi})z =(I_{\ell^2(\mathbb{N})}\otimes I_{\mathcal{H}_0} )z, \quad\forall z \in  \ell^2(\mathbb{N})\otimes \mathcal{H}_0.
	 \end{align*}
	Therefore $(\{B_n\}_{n} ,\{\Phi_n\}_{n})$ is a Riesz weak OVF in $ \mathcal{B}(\mathcal{H}_1, \mathcal{H}_0)$. Thus $(\{B_n\}_{n} ,\{\Phi_n\}_{n})$ is  an orthonormal weak OVF in $ \mathcal{B}(\mathcal{H}_1, \mathcal{H}_0)$.
\end{proof}

For the frames in Hilbert spaces there is a characterization using the standard orthonormal basis $\{e_n\}_n$ for  $\ell^2(\mathbb{N})$ given by Holub \cite{HOLUB}. This states that a sequence $\{\tau_n\}_n$ in  $\mathcal{H}$ is a
frame for $\mathcal{H}$	if and only if there exists a surjective bounded linear operator $T:\ell^2(\mathbb{N}) \to \mathcal{H}$ such that $Te_n=\tau_n$, for all $n \in \mathbb{N}$ (see \cite{HOLUB}). We now derive such a result for factorable weak OVFs.
\begin{theorem}\label{THAFSCHAR}
	A pair  $( \{A_n\}_{n},  \{\Psi_n\}_{n} )$ is a factorable weak OVF  in $ \mathcal{B}(\mathcal{H}, \mathcal{H}_0)$
	if and only if 
	\begin{align*}
	A_n=L_n^* U, \quad \Psi_n=L_n^*V, \quad \forall n \in \mathbb{N},
	\end{align*}  
	where $U, V:\mathcal{H} \rightarrow \ell^2(\mathbb{N}) \otimes \mathcal{H}_0$ are bounded linear operators such that $V^*U$ is bounded invertible.
\end{theorem}
\begin{proof}
	$(\Leftarrow)$ Clearly $\theta_A$ and $\theta_\Psi$ are well-defined bounded linear operators. Let $h\in \mathcal{H}$. Then using Equation (\ref{LEQUATION}), we have 
	\begin{align}\label{ORIGINALEQA}
	S_{A, \Psi}h= \sum_{n=1}^\infty
(L_n^*V)^*L_n^*Uh=V^*\left(\sum_{n=1}^\infty L_nL_n^*\right)Uh=V^*Uh.
	\end{align} 
	Hence $S_{A, \Psi}$ is bounded invertible. \\
	$(\Rightarrow)$ Define $U\coloneqq \sum_{n=1}^\infty L_nA_n$, $V\coloneqq \sum_{n=1}^\infty L_n\Psi_n$. Then
	\begin{align*}
	&L_n^* U=L_n^* \left(\sum_{k=1}^\infty L_kA_k\right)=\sum_{k=1}^\infty L_n^*L_kA_k=A_n,\\
	&L_n^* V=L_n^* \left(\sum_{k=1}^\infty L_k\Psi_k\right)=\sum_{k=1}^\infty L_n^*L_k\Psi_k=\Psi_n, \quad \forall n \in \mathbb{N}
	\end{align*}
	   and 
	\begin{align*}
	V^*U=\left(\sum_{n=1}^\infty \Psi_n^*L_n^*\right)\left(\sum_{k=1}^\infty L_kA_k\right) =  \sum_{n=1}^\infty \Psi_n^*A_n=S_{A, \Psi}
	\end{align*}
	 which is bounded invertible.
\end{proof}
Using Theorem \ref{THAFSCHAR} we can characterize Riesz and orthonormal factorable weak OVFs.
\begin{corollary}\label{FIRSTCOROLLARY}
	A pair  $( \{A_n\}_{n},  \{\Psi_n\}_{n} )$ is a Riesz factorable weak OVF  in $ \mathcal{B}(\mathcal{H}, \mathcal{H}_0)$
if and only if 
\begin{align*}
A_n=L_n^* U, \quad \Psi_n=L_n^*V, \quad \forall n \in \mathbb{N},
\end{align*}  
where $U, V:\mathcal{H} \rightarrow \ell^2(\mathbb{N}) \otimes \mathcal{H}_0$ are bounded linear operators such that $V^*U$ is bounded invertible and $ U(V^*U)^{-1}V^* =I_{\ell^2(\mathbb{N}) \otimes \mathcal{H}_0}$.	
\end{corollary}
\begin{proof}
	$(\Leftarrow)$ $P_{A,\Psi}= U(V^*U)^{-1}V^* =I_{\ell^2(\mathbb{N}) \otimes \mathcal{H}_0}$.
	
	$(\Rightarrow)$ Let $U$ and $V$ be as in Theorem \ref{THAFSCHAR}. Then 
	$U(V^*U)^{-1}V^*=P_{A,\Psi}=I_{\ell^2(\mathbb{N}) \otimes \mathcal{H}_0}$.
\end{proof}
\begin{corollary}
	A pair  $( \{A_n\}_{n},  \{\Psi_n\}_{n} )$ is an orthonormal  factorable weak OVF  in $ \mathcal{B}(\mathcal{H}, \mathcal{H}_0)$
	if and only if 
	\begin{align*}
	A_n=L_n^* U, \quad \Psi_n=L_n^*V, \quad \forall n \in \mathbb{N},
	\end{align*}  
	where $U, V:\mathcal{H} \rightarrow \ell^2(\mathbb{N}) \otimes \mathcal{H}_0$ are bounded linear operators such that $V^*U$ is bounded invertible and $ V^*U=I_\mathcal{H}$, $I_{\ell^2(\mathbb{N}) \otimes \mathcal{H}_0}= UV^*$.	
\end{corollary}
\begin{proof}
	We use Corollary \ref{FIRSTCOROLLARY}. $ (\Leftarrow)$ $S_{A,\Psi}=V^*U=I_\mathcal{H}, P_{A,\Psi}=\theta_AS_{A,\Psi}^{-1}\theta_\Psi^*=\theta_A\theta_\Psi^*= \theta_FUV^*\theta_F^*=I_{\ell^2(\mathbb{N})}\otimes I_{\mathcal{H}_0}.$
	
	$(\Rightarrow)$ $V^*U=S_{A,\Psi}=I_\mathcal{H},$ and by using Proposition \ref{ORTHORESULT},
	\begin{align*}
	UV^*&= \left(\sum_{n=1}^\infty L_n^*A_n\right)\left( \sum_{k=1}^\infty\Psi_k^*L_k\right) =
	 \sum_{n=1}^\infty L_nL_n^*=I_{\ell^2(\mathbb{N}) \otimes \mathcal{H}_0}. 
	\end{align*}
\end{proof}
There is a  variation of result of Holub  given by Christensen \cite{CHRISTENSENBOOK} which states as follows. Let $\{\omega_n\}_n$ be an orthonormal basis for   $\mathcal{H}$. Then  a sequence $\{\tau_n\}_n$ in  $\mathcal{H}$ is a
	frame for $\mathcal{H}$	if and only if there exists a surjective bounded linear operator $T:\mathcal{H} \to \mathcal{H}$ such that $T\omega_n=\tau_n$, for all $n \in \mathbb{N}$.	We now derive such a result for factorable weak OVF. For this, we need to know the definition of orthonormal basis for operators.
\begin{definition}\cite{SUN1}\label{ONBDEFINITIONOVHS}
	A collection  $ \{F_n \}_{n}$ in $ \mathcal{B}(\mathcal{H}, \mathcal{H}_0)$ is said to be an orthonormal basis in  $ \mathcal{B}(\mathcal{H},\mathcal{H}_0)$ if 
	$$\langle F_n ^*y, F_k^*z\rangle=\delta_{n,k}\langle y, z\rangle  , ~\forall y, z \in \mathcal{H}_0, ~\forall n , k \in \mathbb{N} ~ \text{and } ~ \sum\limits_{n=1}^\infty\|F_n h\|^2=\|h\|^2, ~ \forall h \in \mathcal{H}.$$
\end{definition}
We observe $\langle F_n ^*y, F_k^*z\rangle=\delta_{n ,k}\langle y, z\rangle  , \forall y, z \in \mathcal{H}_0, \forall n , k \in \mathbb{N}$ if and only if $F_n F_k^*=\delta_{n ,k}I_{\mathcal{H}_0} , \forall n , k \in \mathbb{N}$. Hence if $ \{F_n \}_{n}$ is an orthonormal basis, then $\|F_n \|^2=\|F_n F_n ^*\|=1, \forall n  \in \mathbb{N} $ and   $ \sum_{n=1}^\infty F_n ^*F_n =I_\mathcal{H}$.
\begin{example}
 Equation (\ref{LEQUATION}) says  that $ \{L^*_n\}_{n}$  is an orthonormal basis in $ \mathcal{B}(\ell^2(\mathbb{N})\otimes\mathcal{H}_0, \mathcal{H}_0)$.
\end{example}
\begin{theorem}\label{SECONDCHAROV}
Let $ \{F_n\}_{n}$ be an  orthonormal basis in $ \mathcal{B}(\mathcal{H},\mathcal{H}_0).$ Then 	a pair  $( \{A_n\}_{n},  \{\Psi_n\}_{n} )$ is a factorable weak OVF  in $ \mathcal{B}(\mathcal{H}, \mathcal{H}_0)$
if and only if 
\begin{align*}
A_n=F_n U, \quad \Psi_n=F_nV, \quad \forall n \in \mathbb{N},
\end{align*}  
 where $ U,V :\mathcal{H}\to \mathcal{H} $ are bounded linear operators such that  $ V^*U$ is bounded  invertible.	
\end{theorem}
\begin{proof}
$(\Leftarrow)$ $ \sum_{n=1}^\infty L_n(F_nU)= (\sum_{n=1}^\infty L_nF_n)U,$ $ \sum_{n=1}^\infty L_n(F_nV)= (\sum_{n=1}^\infty L_nF_n)V.$ These show analysis operators for $ (\{F_nU\}_{n},\{F_nV\}_{n})$ are well-defined bounded linear operators  and the equality $$\sum_{n=1}^\infty(F_nV)^*(F_nU)=V^*U$$
 shows that it is a factorable weak  OVF.

$(\Rightarrow)$ Let $( \{A_n\}_{n},  \{\Psi_n\}_{n} )$ be a factorable weak OVF. Note that the series $  \sum_{n=1}^\infty F_n^*A_n$ and $  \sum_{n=1}^\infty F_n^*\Psi_n$ converge. In fact,  for each $h \in \mathcal{H}$, 
\begin{align*}
\left\|\sum_{n=1}^mF_n ^*A_n h\right\|^2=\left\langle\sum_{n=1}^mF_n^*A_nh, \sum_{k=1}^mF_k^*A_kh\right\rangle= \sum_{n=1}^m\left\langle A_nh, F_n\left(\sum_{k=1}^mF_k^*A_kh\right) \right\rangle=\sum_{n=1}^m\|A_nh\|^2.
\end{align*}
 which  converges to $\| \theta_Ah\|^2=\|\sum_{n=1}^\infty L_nA_nh\|^2=\sum_{n=1}^\infty\|A_nh\|^2$. Define $U\coloneqq \sum_{n=1}^\infty F_n^*A_n$ and $V\coloneqq \sum_{n=1}^\infty F_n^*\Psi_n$. Then $F_nU=A_n, F_nV=\Psi_n ,  \forall n \in \mathbb{N} $ and 
 \begin{align*}
 V^*U=\left(\sum_{n=1}^\infty\Psi_n^*F_n\right)\left(\sum_{k=1}^\infty F_k^*A_k\right)=\sum_{n=1}^\infty\Psi_n^*A_n=S_{A,\Psi}
 \end{align*}
  which is bounded invertible.	
\end{proof}	 
\begin{corollary}
Let $ \{F_n\}_{n}$ be an  orthonormal basis in $ \mathcal{B}(\mathcal{H},\mathcal{H}_0).$ Then 	a pair  $( \{A_n\}_{n},  \{\Psi_n\}_{n} )$ is  		
	\begin{enumerate}[\upshape(i)]
		\item a Riesz factorable weak OVF  in $ \mathcal{B}(\mathcal{H}, \mathcal{H}_0)$
		if and only if 
		\begin{align*}
		A_n=F_n U, \quad \Psi_n=F_nV, \quad \forall n \in \mathbb{N},
		\end{align*}  
		where $ U,V :\mathcal{H}\to \mathcal{H} $ are bounded linear operators such that  $ V^*U$ is bounded  invertible and $ U(V^*U)^{-1}V^* =I_{\mathcal{H}}$.	
		\item an orthonormal  factorable weak OVF  in $ \mathcal{B}(\mathcal{H}, \mathcal{H}_0)$
		if and only if 
		\begin{align*}
		A_n=F_n U, \quad \Psi_n=F_nV, \quad \forall n \in \mathbb{N},
		\end{align*}  
		where $ U,V :\mathcal{H}\to \mathcal{H} $ are bounded linear operators such that  $ V^*U$ is bounded  invertible and $ V^*U=I_\mathcal{H}= UV^*$. 	
	\end{enumerate}
\end{corollary}
\begin{proof}
\begin{enumerate}[\upshape(i)]
	\item $ (\Leftarrow)$ $ P_{A,\Psi}=\theta_AS_{A,\Psi}^{-1}\theta_\Psi^*=(\sum_{n=1}^\infty L_nF_nU)(V^*U)^{-1}(\sum_{k=1}^\infty V^*F^*_kL_k^*)=\theta_FU(V^*U)^{-1}V^*\theta_F^*=\theta_FI_\mathcal{H}\theta_F^* =\sum_{n=1}^\infty L_nF_n(\sum_{k=1}^\infty F_k^*L^*_k)=\sum_{n=1}^\infty L_nL_n^*=I_{\ell^2(\mathbb{N})}\otimes I_{\mathcal{H}_0}$. 
	
	$ (\Rightarrow)$ Let $U$ and $V$ be as in Theorem \ref{SECONDCHAROV}. Then 
	\begin{align*}
	U(V^*U)^{-1}V^*&=\left(\sum_{k=1}^\infty F_k^*A_k\right)S_{A,\Psi}^{-1}\left(\sum_{n=1}^\infty \Psi_n^*F_n\right)\\
	&=\left(\sum_{r=1}^\infty F_r^*L_r^*\right)\left(\sum_{k=1}^\infty L_kA_k\right)S_{A,\Psi}^{-1}\left(\sum_{n=1}^\infty \Psi_n^*L_n^*\right)\left(\sum_{m=1}^\infty L_mF_m\right)\\
	&=\left(\sum_{r=1}^\infty F_r^*L_r^*\right)\theta_AS_{A,\Psi}^{-1}\theta_\Psi^*\left(\sum_{m=1}^\infty L_mF_m\right)=\left(\sum_{r=1}^\infty F_r^*L_r^*\right)P_{A,\Psi}\left(\sum_{m=1}^\infty L_mF_m\right) \\
	&=\left(\sum_{r=1}^\infty F_r^*L_r^*\right)(I_{\ell^2(\mathbb{N})}\otimes I_{\mathcal{H}_0})\left(\sum_{m=1}^\infty L_mF_m\right) =\left(\sum_{r=1}^\infty F_r^*L_r^*\right)\left(\sum_{m=1}^\infty L_mF_m\right) \\
	&=\sum_{r=1}^\infty F_r^*F_r=I_{\mathcal{H}}.
	\end{align*}
	\item We use (i). $ (\Leftarrow)$ $S_{A,\Psi}=V^*U=I_\mathcal{H}, P_{A,\Psi}=\theta_AS_{A,\Psi}^{-1}\theta_\Psi^*=\theta_A\theta_\Psi^*= \theta_FUV^*\theta_F^*=\theta_FI_\mathcal{H}\theta_F^*=I_{\ell^2(\mathbb{N})}\otimes I_{\mathcal{H}_0}.$
	
	$(\Rightarrow)$ $V^*U=S_{A,\Psi}=I_\mathcal{H}$ and using Proposition \ref{ORTHORESULT}, $	UV^*= \left(\sum_{n=1}^\infty F_n^*A_n\right)\left( \sum_{k=1}^\infty\Psi_k^*F_k\right) =\sum_{n=1}^\infty F_n^*F_n=I_\mathcal{H}$.
	\end{enumerate}	
\end{proof}
We next derive another characterization which is free from natural numbers.
\begin{theorem}\label{OPERATORCHARACTERIZATIONHILBERT2}
	Let $\{A_n\}_{n},\{\Psi_n\}_{n}$ be in $ \mathcal{B}(\mathcal{H}, \mathcal{H}_0).$ Then  $( \{A_n\}_{n},  \{\Psi_n\}_{n} )$  is a factorable weak OVF  
	\begin{enumerate}[\upshape(i)]
		\item   if and only if $$U:\ell^2(\mathbb{N})\otimes \mathcal{H}_0 \ni y\mapsto\sum\limits_{n=1}^\infty A_n^*L_n^*y \in \mathcal{H}, ~\text{and} ~ V:\ell^2(\mathbb{N})\otimes \mathcal{H}_0 \ni z\mapsto\sum\limits_{n=1}^\infty \Psi_n^*L^*_nz \in \mathcal{H} $$ 
		are well-defined bounded linear operators   such that  $  VU^*$ is bounded invertible.  
		\item    if and only if $$U:\ell^2(\mathbb{N})\otimes \mathcal{H}_0 \ni y\mapsto\sum\limits_{n=1}^\infty A_n^*L_n^*y \in \mathcal{H}, ~\text{and} ~ S: \mathcal{H} \ni g\mapsto \sum\limits_{n=1}^\infty L_n\Psi_ng \in \ell^2(\mathbb{N})\otimes \mathcal{H}_0 $$ 
		are well-defined bounded linear operators  such that  $  S^*U^*$  is bounded invertible.
		\item  if and only if  $$R:   \mathcal{H} \ni h\mapsto \sum\limits_{n=1}^\infty L_nA_nh \in \ell^2(\mathbb{N})\otimes \mathcal{H}_0, ~\text{and} ~ V: \ell^2(\mathbb{N})\otimes \mathcal{H}_0 \ni z\mapsto\sum\limits_{n=1}^\infty \Psi_n^*L_n^*z \in \mathcal{H} $$ 
		are well-defined bounded linear operators   such that  $  VR$ is bounded invertible.
		\item  if and only if  $$ R:   \mathcal{H} \ni h\mapsto \sum\limits_{n=1}^\infty L_nA_nh \in \ell^2(\mathbb{N})\otimes \mathcal{H}_0, ~\text{and} ~  S:   \mathcal{H} \ni g\mapsto \sum\limits_{n=1}^\infty L_n\Psi_ng \in \ell^2(\mathbb{N})\otimes \mathcal{H}_0 $$
		are well-defined bounded linear operators  such that  $ S^*R $ is bounded invertible.  
	\end{enumerate}
	
\end{theorem} 
\begin{proof}
	We prove (i) and others are similar.  $(\Rightarrow)$ Now $U=\theta_A^*$, $V=\theta_\Psi^*$ and $VU^*=\theta_\Psi^*\theta_A=S_{A,\Psi}$.
	
	$(\Leftarrow)$ Now $\theta_A=U^*$, $\theta_\Psi=V^*$ and $S_{A,\Psi}=\theta_\Psi^*\theta_A=VU^*$.
\end{proof}

Now we try to characterize all dual OVFs. For the case of Hilbert spaces, this was done by Li \cite{LI}.
\begin{lemma}\label{FIRSTLEMMA}
Let $( \{A_n\}_{n},  \{\Psi_n\}_{n} )$ is a factorable weak OVF in    $\mathcal{B}(\mathcal{H}, \mathcal{H}_0)$. Then a factorable weak  OVF  $ (\{B_n\}_{n} , \{\Phi_n\}_{n} )$  in $\mathcal{B}(\mathcal{H}, \mathcal{H}_0)$	is a dual for $( \{A_n\}_{n},  \{\Psi_n\}_{n} )$ if and only if 
\begin{align*}
B_n=L_n^* U, \quad \Phi_n=L_n^*V^*, \quad \forall n \in \mathbb{N}
\end{align*}
where $U:\mathcal{H} \rightarrow \ell^2(\mathbb{N}) \otimes \mathcal{H}_0$  is a bounded right-inverse of $\theta_\Psi^* $, $V: \ell^2(\mathbb{N}) \otimes \mathcal{H}_0\to \mathcal{H}$ is a bounded left-inverse of $\theta_A $ such that $VU$ is bounded invertible.
\end{lemma}
\begin{proof}
$(\Leftarrow)$  ``If" part of proof of Theorem \ref{THAFSCHAR}, says that $ (\{B_n\}_{n} , \{\Phi_n\}_{n} )$ is a factorable weak OVF in  $\mathcal{B}(\mathcal{H}, \mathcal{H}_0)$. We now  check for the duality of $ (\{B_n\}_{n} , \{\Phi_n\}_{n} )$. Consider $\theta_\Phi^*\theta_A=V^* \theta_A=I_\mathcal{H} $, $ \theta_\Psi^*\theta_B=\theta_\Psi^* U =I_\mathcal{H}$.\\
$(\Rightarrow)$ Let $ (\{B_n\}_{n} , \{\Phi_n\}_{n} )$  be a dual factorable weak OVF for  $( \{A_n\}_{n},  \{\Psi_n\}_{n} )$.  Then $\theta_\Psi^*\theta_B =I_\mathcal{H}= \theta_\Phi^*\theta_A $. Define $ U\coloneqq\theta_B, V\coloneqq\theta_\Phi^*.$ Then $U:\mathcal{H} \rightarrow \ell^2(\mathbb{N}) \otimes \mathcal{H}_0$ is a bounded  right-inverse of $\theta_\Psi^* $, $V: \ell^2(\mathbb{N}) \otimes \mathcal{H}_0\to \mathcal{H}$ is a left inverse of $\theta_A $ such that $VU=\theta_\Phi^*\theta_B=S_{B,\Phi}$ is bounded invertible. We now see 
\begin{align*}
L_n^* U=L_n^*\left(\sum\limits_{k=1}^\infty L_kB_k\right)=B_n, \quad L_n^*V^*=L_n^*\left(\sum\limits_{k=1}^\infty L_k\Phi_k\right)=\Phi_n, \quad \forall n \in \mathbb{N}.
\end{align*}
\end{proof}
\begin{lemma}\label{SECONDLEMMA}
Let $( \{A_n\}_{n},  \{\Psi_n\}_{n} )$ be a  factorable weak OVF in    $\mathcal{B}(\mathcal{H}, \mathcal{H}_0)$. Then  
\begin{enumerate}[\upshape(i)]
	\item $R:\mathcal{H} \to \ell^2(\mathbb{N})\otimes\mathcal{H}_0 $ is a bounded right-inverse of $ \theta_\Psi^*$ if and only if 
	\begin{align*}
	R=\theta_AS_{A,\Psi}^{-1}+(I_{\ell^2(\mathbb{N})\otimes\mathcal{H}_0}-\theta_AS_{A,\Psi}^{-1}\theta_\Psi^*)U,
	\end{align*}
	where $U :\mathcal{H} \to \ell^2(\mathbb{N})\otimes\mathcal{H}_0$ is a bounded linear operator.
\item $L:\ell^2(\mathbb{N})\otimes\mathcal{H}_0\rightarrow \mathcal{H} $ is a bounded left-inverse of $ \theta_A$ if and only if 
	\begin{align*}
	L=S_{A,\Psi}^{-1}\theta_\Psi^*+V(I_{\ell^2(\mathbb{N})\otimes\mathcal{H}_0}-\theta_A S_{A,\Psi}^{-1}\theta_\Psi^*),
	\end{align*}  
	 where $V:\ell^2(\mathbb{N})\otimes\mathcal{H}_0\to\mathcal{H}$ is a bounded linear operator.
\end{enumerate}		
\end{lemma}
\begin{proof}
\begin{enumerate}[\upshape (i)]
	\item $(\Leftarrow)$ Let $U :\mathcal{H} \to \ell^2(\mathbb{N})\otimes\mathcal{H}_0$ be a bounded linear operator. Then $ \theta_\Psi^*(\theta_AS_{A,\Psi}^{-1}+(I_{\ell^2(\mathbb{N})\otimes\mathcal{H}_0}-\theta_AS_{A,\Psi}^{-1}\theta_\Psi^*)U)=I_\mathcal{H}+\theta_\Psi^*U-\theta_\Psi^*U=I_\mathcal{H}$. Therefore  $\theta_AS_{A,\Psi}^{-1}+(I_{\ell^2(\mathbb{N})\otimes\mathcal{H}_0}-\theta_AS_{A,\Psi}^{-1}\theta_\Psi^*)U$ is a bounded right-inverse of $\theta_\Psi^*$.
	
	$(\Rightarrow)$ Let $R:\mathcal{H} \to \ell^2(\mathbb{N})\otimes\mathcal{H}_0 $ be a bounded right-inverse of $ \theta_\Psi^*$. Define $U\coloneqq R$. Then $	\theta_AS_{A,\Psi}^{-1}+(I_{\ell^2(\mathbb{N})\otimes\mathcal{H}_0}-\theta_AS_{A,\Psi}^{-1}\theta_\Psi^*)U=	\theta_AS_{A,\Psi}^{-1}+(I_{\ell^2(\mathbb{N})\otimes\mathcal{H}_0}-\theta_AS_{A,\Psi}^{-1}\theta_\Psi^*)R=\theta_AS_{A,\Psi}^{-1}+R-\theta_AS_{A,\Psi}^{-1}=R$.
	\item $(\Leftarrow)$ Let $V: \ell^2(\mathbb{N})\otimes\mathcal{H}_0\rightarrow \mathcal{H}$ be a bounded linear operator. Then $(S_{A,\Psi}^{-1}\theta_\Psi^*+V(I_{\ell^2(\mathbb{N})\otimes\mathcal{H}_0}-\theta_A S_{A,\Psi}^{-1}\theta_\Psi^*))\theta_A=I_\mathcal{H}+V\theta_A-V\theta_A I_\mathcal{H}=I_\mathcal{H}$. Therefore  $S_{A,\Psi}^{-1}\theta_\Psi^*+V(I_{\ell^2(\mathbb{N})\otimes\mathcal{H}_0}-\theta_A S_{A,\Psi}^{-1}\theta_\Psi^*)$ is a bounded left-inverse of $\theta_A$.
	
	$(\Rightarrow)$ Let $ L:\ell^2(\mathbb{N})\otimes\mathcal{H}_0\rightarrow \mathcal{H}$ be a bounded left-inverse of $ \theta_A$. Define $V\coloneqq L$. Then $S_{A,\Psi}^{-1}\theta_\Psi^*+V(I_{\ell^2(\mathbb{N})\otimes\mathcal{H}_0}-\theta_A S_{A,\Psi}^{-1}\theta_\Psi^*) =S_{A,\Psi}^{-1}\theta_\Psi^*+L(I_{\ell^2(\mathbb{N})\otimes\mathcal{H}_0}-\theta_A S_{A,\Psi}^{-1}\theta_\Psi^*)=S_{A,\Psi}^{-1}\theta_\Psi^*+L-I_{\mathcal{H}}S_{A,\Psi}^{-1}\theta_\Psi^*= L$. 	
\end{enumerate}	
 \end{proof}
\begin{theorem}
Let $( \{A_n\}_{n},  \{\Psi_n\}_{n} )$ be a  factorable  weak OVF in $\mathcal{B}(\mathcal{H}, \mathcal{H}_0)$. Then a  factorable weak OVF  $ (\{B_n\}_{n} , \{\Phi_n\}_{n} )$   in $\mathcal{B}(\mathcal{H}, \mathcal{H}_0)$ is a dual  for $( \{A_n\}_{n},  \{\Psi_n\}_{n} )$ if and only if
\begin{align*}
&B_n=A_nS_{A,\Psi}^{-1}+L_n^*U-A_nS_{A,\Psi}^{-1}\theta_\Psi^*U,\\
&\Phi_n=\Psi_n(S_{A,\Psi}^{-1})^*+L_n^*V^*-\Psi_n (S_{A,\Psi}^{-1})^*\theta_A^*V^*, \quad \forall n \in \mathbb{N}
\end{align*}
such that the operator 
\begin{align*}
S_{A, \Psi}^{-1}+VU-V\theta_AS_{A, \Psi}^{-1}\theta_\Psi^* U
\end{align*}
is bounded invertible, where   $U :\mathcal{H} \to \ell^2(\mathbb{N})\otimes\mathcal{H}_0$  and $V:\ell^2(\mathbb{N})\otimes\mathcal{H}_0\to\mathcal{H}$ are bounded linear operators.	
\end{theorem}
\begin{proof}
	Lemmas \ref{FIRSTLEMMA} and  \ref{SECONDLEMMA}
	 give the characterization of dual weak OVF as 
	\begin{align*}
	&B_n=L_n^*(\theta_AS_{A,\Psi}^{-1}+(I_{\ell^2(\mathbb{N})\otimes\mathcal{H}_0}-\theta_AS_{A,\Psi}^{-1}\theta_\Psi^*)U)=A_nS_{A,\Psi}^{-1}+L_n^*U-A_nS_{A,\Psi}^{-1}\theta_\Psi^*U,\\
	&\Phi_n=L_n^*(\theta_\Psi(S_{A,\Psi}^{-1})^*+(I_{\ell^2(\mathbb{N})\otimes\mathcal{H}_0}-\theta_\Psi (S_{A,\Psi}^{-1})^*\theta_A^*)V^*)=\Psi_n(S_{A,\Psi}^{-1})^*+L_n^*V^*-\Psi_n (S_{A,\Psi}^{-1})^*\theta_A^*V^*, \quad \forall n \in \mathbb{N}
	\end{align*}
	such that the operator 
	$$(S_{A,\Psi}^{-1}\theta_\Psi^*+V(I_{\ell^2(\mathbb{N})\otimes\mathcal{H}_0}-\theta_A S_{A,\Psi}^{-1}\theta_\Psi^*))(\theta_AS_{A,\Psi}^{-1}+(I_{\ell^2(\mathbb{N})\otimes\mathcal{H}_0}-\theta_AS_{A,\Psi}^{-1}\theta_\Psi^*)U) $$
	is bounded invertible, where $U :\mathcal{H} \to \ell^2(\mathbb{N})\otimes\mathcal{H}_0$ and $V:\ell^2(\mathbb{N})\otimes\mathcal{H}_0\to\mathcal{H}$ are bounded linear operators. We expand and  get 
	\begin{align*}
	&(S_{A,\Psi}^{-1}\theta_\Psi^*+V(I_{\ell^2(\mathbb{N})\otimes\mathcal{H}_0}-\theta_A S_{A,\Psi}^{-1}\theta_\Psi^*))(\theta_AS_{A,\Psi}^{-1}+(I_{\ell^2(\mathbb{N})\otimes\mathcal{H}_0}-\theta_AS_{A,\Psi}^{-1}\theta_\Psi^*)U)\\
	&=S_{A, \Psi}^{-1}+VU-V\theta_AS_{A, \Psi}^{-1}\theta_\Psi^* U.
	\end{align*}
\end{proof}
Like similarity, there is another notion called as orthogonality for frames in Hilbert spaces. This was introduced by Balan \cite{BALANTHESIS} and later studied by Han and Larson \cite{HANLARSON} (see \cite{KORNELSON}). We now define the orthogonality for weak OVFs.
\begin{definition}
	A weak  OVF  $ (\{B_n\}_{n} , \{\Phi_n\}_{n} )$  in $\mathcal{B}(\mathcal{H}, \mathcal{H}_0)$ is said to be orthogonal to a weak  OVF  $  ( \{A_n\}_{n},  \{\Psi_n\}_{n} ) $ in $\mathcal{B}(\mathcal{H}, \mathcal{H}_0)$ if 
	\begin{align*}
	 \sum_{n=1}^\infty\Psi_n^*B_n= \sum_{n=1}^\infty\Phi^*_nA_n=0.
	\end{align*}
\end{definition}  
Remarkable property of orthogonal frames is that we can interpolate as well as  we can take direct sum of them to get new frames. These are illustrated in the following two results.
\begin{proposition}
	Let $  ( \{A_n\}_{n},  \{\Psi_n\}_{n} ) $ and $ (\{B_n\}_{n} , \{\Phi_n\}_{n} )$ be  two Parseval  OVFs in   $\mathcal{B}(\mathcal{H}, \mathcal{H}_0)$ which are  orthogonal. If $C,D,E,F \in \mathcal{B}(\mathcal{H})$ are such that $ C^*E+D^*F=I_\mathcal{H}$, then  
	\begin{align*}
	 (\{A_nC+B_nD\}_{n}, \{\Psi_nE+\Phi_nF\}_{n})
	\end{align*}
	 is a  Parseval weak  OVF in  $\mathcal{B}(\mathcal{H}, \mathcal{H}_0)$. In particular,  if scalars $ c,d,e,f$ satisfy $\bar{c}e+\bar{d}f =1$, then $ (\{cA_n+dB_n\}_{n}, \{e\Psi_n+f\Phi_n\}_{n}) $ is   a Parseval weak  OVF.
\end{proposition} 
\begin{proof}
	We use the definition of frame operator and get 
	\begin{align*}
	S_{AC+BD,\Psi E+\Phi F} &=\sum_{n=1}^\infty(\Psi_nE+\Phi_nF)^*(A_nC+B_nD)\\
	&=E^*S_{A,\Psi}C+E^*\left(\sum_{n=1}^\infty\Psi_n^*B_n\right)D+F^*\left(\sum_{n=1}^\infty\Phi_n^*A_n\right)C+F^*S_{B,\Phi}D\\
	&=E^*I_\mathcal{H}C+E^*0D+F^*0C+F^*I_\mathcal{H}D=I_\mathcal{H}.
	 \end{align*}
\end{proof}

\begin{proposition}
	If $  ( \{A_n\}_{n},  \{\Psi_n\}_{n} ) $  and $ (\{B_n\}_{n} , \{\Phi_n\}_{n} )$ are   orthogonal weak OVFs in $ \mathcal{B}(\mathcal{H}, \mathcal{H}_0)$, then  $(\{A_n\oplus B_n\}_{n},\{\Psi_n\oplus \Phi_n\}_{n})$ is a  weak  OVF in $ \mathcal{B}(\mathcal{H}\oplus \mathcal{H}, \mathcal{H}_0).$    Further, if both $  ( \{A_n\}_{n},  \{\Psi_n\}_{n} ) $  and $ (\{B_n\}_{n} , \{\Phi_n\}_{n} )$ are  Parseval, then $(\{A_n\oplus B_n\}_{n},\{\Psi_n\oplus \Phi_n\}_{n})$ is Parseval.
\end{proposition}
\begin{proof}
	Let $ h \oplus g \in \mathcal{H}\oplus \mathcal{H}$. Then 
	\begin{align*}
	S_{A\oplus B, \Psi\oplus \Phi}(h\oplus g)&=\sum_{n=1}^\infty(\Psi_n\oplus \Phi_n)^*(A_n\oplus B_n)(h\oplus g)=\sum_{n=1}^\infty(\Psi_n\oplus \Phi_n)^*(A_nh+ B_ng)\\
	&=\sum_{n=1}^\infty(\Psi_n^*(A_nh+B_ng)\oplus \Phi_n^*(A_nh+B_ng))
	\\
	&=\left(\sum_{n=1}^\infty\Psi_n^*A_nh+\sum_{n=1}^\infty\Psi_n^*B_ng\right)\oplus \left(\sum_{n=1}^\infty\Phi_n^*A_nh+\sum_{n=1}^\infty\Phi_n^*B_ng\right)\\
	&=(S_{A,\Psi}h+0)\oplus(0+S_{B,\Phi}g) =(S_{A,\Psi}\oplus S_{B,\Phi})(h\oplus g).
	\end{align*}	
\end{proof}  
\section{Equivalence}\label{SIMILARITYCOMPOSITIONANDTENSORPRODUCT}
Balan introduced the notion of similarity or equivalence for frames for Hilbert spaces \cite{RADU}. In \cite{KAFTAL}, this notion was defined for operator-valued frames. We now define a similar notion for factorable weak OVFs.
\begin{definition}
	A factorable weak OVF   $( \{B_n\}_{n},  \{\Phi_n\}_{n} ) $  in $ \mathcal{B}(\mathcal{H}, \mathcal{H}_0)$    is said to be similar  to a factorable  weak OVF  $  ( \{A_n\}_{n},  \{\Psi_n\}_{n} ) $ in $ \mathcal{B}(\mathcal{H}, \mathcal{H}_0)$  if there exist bounded  invertible  $ R_{A,B}, R_{\Psi, \Phi} \in \mathcal{B}(\mathcal{H})$   such that 
	\begin{align*}
	B_n=A_nR_{A,B} , \quad \Phi_n=\Psi_nR_{\Psi, \Phi}, \quad \forall n \in \mathbb{N}.
	\end{align*} 
\end{definition}
Since  $ R_{A,B}, R_{\Psi, \Phi}$ are bounded invertible, it easily follows that the notion similarity is symmetric. We further have that the relation ``similarity" is an equivalence relation on the set 
\begin{align*}
\{( \{A_n\}_{n},  \{\Psi_n\}_{n} ):( \{A_n\}_{n},  \{\Psi_n\}_{n} ) \text{  is a factorable weak OVF}\}.
\end{align*} 
Similar frames have nice property that knowing analysis, synthesis and frame operators of one give that of another.
\begin{lemma}\label{SIM}
	Let $  ( \{A_n\}_{n},  \{\Psi_n\}_{n} ) $ and  $  ( \{B_n\}_{n},  \{\Phi_n\}_{n} ) $ be similar factorable weak OVFs  and   $B_n=A_nR_{A,B} ,\Phi_n=\Psi_nR_{\Psi, \Phi},  \forall n \in \mathbb{N}$, for some invertible $ R_{A,B} ,R_{\Psi, \Phi} \in \mathcal{B}(\mathcal{H}).$ Then 
	\begin{enumerate}[\upshape(i)]
		\item $ \theta_B=\theta_A R_{A,B}, \theta_\Phi=\theta_\Psi R_{\Psi,\Phi}$.
		\item $S_{B,\Phi}=R_{\Psi,\Phi}^*S_{A, \Psi}R_{A,B}$.
		\item $P_{B,\Phi}=P_{A, \Psi}.$
	\end{enumerate}
\end{lemma}
\begin{proof}
$ \theta_B=\sum_{n=1}^\infty L_nB_n=\sum_{n=1}^\infty L_nA_nR_{A,B}=\theta_AR_{A,B} $. Similarly $ \theta_\Phi=\theta_\Psi R_{\Psi,\Phi}$. Now using operators $\theta_B$ and $\theta_\Phi$ we get $S_{B,\Phi}=\sum_{n=1}^\infty \Phi_n^*B_n=\sum_{n=1}^\infty(\Psi_nR_{\Psi,\Phi})^*(A_nR_{A,B})=R_{\Psi, \Phi}^*\left (\sum_{n=1}^\infty\Psi_n^*A_n\right )R_{A,B}=R_{\Psi,\Phi}^*S_{A, \Psi}R_{A,B}$. We now use (i) and (ii) to get 
\begin{align*}
P_{B,\Phi}=\theta_BS_{B,\Phi}^{-1}\theta_\Phi^*=(\theta_AR_{A,B})(R_{\Psi,\Phi}^*S_{A, \Psi}R_{A,B})^{-1}(\theta_\Psi R_{\Psi,\Phi})^*=P_{A,\Psi}.
\end{align*}
\end{proof}
We now classify similarity using operators.
\begin{theorem}\label{RIGHTSIMILARITY}
For two factorable weak OVFs  $  ( \{A_n\}_{n},  \{\Psi_n\}_{n} ) $ and  $  ( \{B_n\}_{n},  \{\Phi_n\}_{n} ) $, the following are equivalent.
	\begin{enumerate}[\upshape(i)]
		\item $B_n=A_nR_{A,B} , \Phi_n=\Psi_nR_{\Psi, \Phi} ,  \forall n \in \mathbb{N},$ for some invertible  $ R_{A,B} ,R_{\Psi, \Phi} \in \mathcal{B}(\mathcal{H}). $
		\item $\theta_B=\theta_AR_{A,B} , \theta_\Phi=\theta_\Psi R_{\Psi, \Phi} $ for some invertible  $ R_{A,B} ,R_{\Psi, \Phi} \in \mathcal{B}(\mathcal{H}). $
		\item $P_{B,\Phi}=P_{A,\Psi}.$
	\end{enumerate}
	If one of the above conditions is satisfied, then  invertible operators in  $ \operatorname{(i)}$ and  $ \operatorname{(ii)}$ are unique and are given by $R_{A,B}=S_{A,\Psi}^{-1}\theta_\Psi^*\theta_B$,   $R_{\Psi, \Phi}=(S_{A,\Psi}^{-1})^*\theta_A^*\theta_\Phi.$ In the case that $  ( \{A_n\}_{n},  \{\Psi_n\}_{n} ) $ is Parseval, then $  ( \{B_n\}_{n},  \{\Phi_n\}_{n} ) $ is  Parseval if and only if $R_{\Psi, \Phi}^*R_{A,B}=I_\mathcal{H} $   if and only if $R_{A,B}R_{\Psi, \Phi}^*=I_\mathcal{H} $.
\end{theorem}
\begin{proof}
	The implications (i) $\Rightarrow$ (ii) $\Rightarrow$ (iii) follow from Lemma \ref{SIM}.  Assume (ii) holds. We show (i) holds. Using Equation (\ref{LEQUATION}), $ B_n=L_n^*\theta_B=L_n^*\theta_AR_{A,B}'=A_nR_{A,B}'$; the same procedure gives $ \Phi_n$ also.   Assume (iii). We note the following $ \theta_B=P_{B,\Phi}\theta_B$ and $ \theta_\Phi=P_{B,\Phi}^*\theta_\Phi.$ Using these, $ \theta_B=P_{A,\Psi}\theta_B=\theta_A(S_{A,\Psi}^{-1}\theta_\Psi^*\theta_B)$ and $ \theta_\Phi=P_{A,\Psi}^*\theta_\Phi=(\theta_AS_{A,\Psi}^{-1}\theta_\Psi^*)^*\theta_\Phi=\theta_\Psi((S_{A,\Psi}^{-1})^*\theta_A^*\theta_\Phi).$ We now try to show that both $S_{A,\Psi}^{-1}\theta_\Psi^*\theta_B$  and $(S_{A,\Psi}^{-1})^*\theta_A^*\theta_\Phi$ are invertible. This is achieved via, 
	\begin{align*}
	(S_{A,\Psi}^{-1}\theta_\Psi^*\theta_B)(S_{B,\Phi}^{-1}\theta_\Phi^*\theta_A)=S_{A,\Psi}^{-1}\theta_\Psi^*P_{B,\Phi}\theta_A= S_{A,\Psi}^{-1}\theta_\Psi^*P_{A,\Psi}\theta_A= S_{A,\Psi}^{-1}\theta_\Psi^*\theta_A=I_\mathcal{H},\\
	( S_{B,\Phi}^{-1}\theta_\Phi^*\theta_A)(S_{A,\Psi}^{-1}\theta_\Psi^*\theta_B)= S_{B,\Phi}^{-1}\theta_\Phi^*P_{A,\Psi}\theta_B=S_{B,\Phi}^{-1}\theta_\Phi^*P_{B,\Phi}\theta_B=S_{B,\Phi}^{-1}\theta_\Phi^*\theta_B=I_\mathcal{H}
	\end{align*}
	and 
	\begin{align*}
	((S_{A,\Psi}^{-1})^*\theta_A^*\theta_\Phi)((S_{B,\Phi}^{-1})^*\theta_B^*\theta_\Psi)=(S_{A,\Psi}^{-1})^*\theta_A^*P_{B,\Phi}^*\theta_\Psi=(S_{A,\Psi}^{-1})^*\theta_A^*P_{A,\Psi}^*\theta_\Psi=(S_{A,\Psi}^{-1})^*\theta_A^*\theta_\Psi=I_\mathcal{H},\\
	((S_{B,\Phi}^{-1})^*\theta_B^*\theta_\Psi)((S_{A,\Psi}^{-1})^*\theta_A^*\theta_\Phi)=(S_{B,\Phi}^{-1})^*\theta_B^*P_{A,\Psi}^* \theta_\Phi=(S_{B,\Phi}^{-1})^*\theta_B^*P_{B,\Phi}^* \theta_\Phi=(S_{B,\Phi}^{-1})^*\theta_B^* \theta_\Phi= I_\mathcal{H}.
	\end{align*} 
	
	Let $ R_{A,B}, R_{\Psi,\Phi} \in \mathcal{B}(\mathcal{H}) $ be invertible. From the previous arguments, $ R_{A,B}$ and $R_{\Psi,\Phi} $ satisfy (i) if and only if  they satisfy (ii). Let $B_n=A_nR_{A,B} , \Phi_n=\Psi_nR_{\Psi, \Phi} ,  \forall n \in \mathbb{N}.$ Using (ii), $\theta_B=\theta_AR_{A,B} , \theta_\Phi=\theta_\Psi R_{\Psi, \Phi}$  $\implies$  $\theta_\Psi^*\theta_B=\theta_\Psi^*\theta_AR_{A,B}=S_{A,\Psi}R_{A,B} , \theta_A^*\theta_\Phi=\theta_A^*\theta_\Psi R_{\Psi, \Phi}=S_{A,\Psi}^*R_{\Psi, \Phi}$. These imply the formula for $R_{A,B}$ and $ R_{\Psi, \Phi}.$ For the last,  we recall $ S_{B,\Phi}=R_{\Psi,\Phi}^*S_{A, \Psi}R_{A,B}$.  
\end{proof}
\begin{corollary}
	For any given factorable weak OVF $  ( \{A_n\}_{n},  \{\Psi_n\}_{n} ) $, the canonical dual of $  ( \{A_n\}_{n},  \{\Psi_n\}_{n} ) $ is the only dual factorable weak OVF  that is similar to $  ( \{A_n\}_{n},  \{\Psi_n\}_{n} ) $.
\end{corollary}
\begin{proof}
	Let  $ (\{B_n\}_{n} , \{\Phi_n\}_{n} )$ be a factorable  weak OVF which is both dual and similar for  $  ( \{A_n\}_{n},  \{\Psi_n\}_{n} ) $.  Then  we have $ \theta_B^*\theta_\Psi=I_\mathcal{H}=\theta_\Phi^*\theta_A$ and  there exist invertible $ R_{A,B},R_{\Psi,\Phi}\in \mathcal{B}(\mathcal{H})$ such that  $B_n=A_nR_{A,B} , \Phi_n=\Psi_nR_{\Psi, \Phi} ,  \forall n \in \mathbb{N} $. Theorem \ref{RIGHTSIMILARITY} gives  $R_{A,B}=S_{A,\Psi}^{-1}\theta_\Psi^*\theta_B, R_{\Psi, \Phi}=S_{A,\Psi}^{-1}\theta_A^*\theta_\Phi.$ But then $R_{A,B}=S_{A,\Psi}^{-1}I_\mathcal{H}=S_{A,\Psi}^{-1}$,  $ R_{\Psi, \Phi}=(S_{A,\Psi}^{-1})^*I_\mathcal{H}=(S_{A,\Psi}^{-1})^*.$ Therefore  $ (\{B_n\}_{n} , \{\Phi_n\}_{n} )$ is the  canonical  dual for  $  ( \{A_n\}_{n},  \{\Psi_n\}_{n} ) $.
\end{proof}
\begin{corollary}
	Two similar factorable weak OVF   cannot be orthogonal.
\end{corollary}
\begin{proof}
	Let a factorable weak OVF $  ( \{B_n\}_{n},  \{\Phi_n\}_{n} ) $ be similar to $  ( \{A_n\}_{n},  \{\Psi_n\}_{n} ) $. Choose invertible $ R_{A,B},R_{\Psi,\Phi}\in \mathcal{B}(\mathcal{H})$ such that  $B_n=A_nR_{A,B} , \Phi_n=\Psi_nR_{\Psi, \Phi} ,  \forall n \in \mathbb{N} $. Using 	Theorem \ref{RIGHTSIMILARITY} and the invertibility of $R_{A,B}^* $ and $S_{A,\Psi}^* $, we get 
	\begin{align*}
	\theta_B^*\theta_\Psi=(\theta_AR_{A,B})^*\theta_\Psi=R_{A,B}^*\theta_A^*\theta_\Psi=R_{A,B}^*S_{A,\Psi}^*\neq 0.
	\end{align*}
\end{proof}

	For every factorable weak OVF $  ( \{A_n\}_{n},  \{\Psi_n\}_{n} ) $, each  of `OVFs'  $( \{A_nS_{A, \Psi}^{-1}\}_{n}, \{\Psi_n\}_{n})$    and  $ (\{A_n \}_{n}, \{\Psi_n(S_{A,\Psi}^{-1})^*\}_{n})$ is a Parseval OVF which is similar to  $  ( \{A_n\}_{n},  \{\Psi_n\}_{n} ) $.  Thus every OVF is similar to  Parseval OVFs.

\section{Frames and discrete group representations} \label{FRAMESANDDISCRETEGROUPREPRESENTATIONS}
Let $ G$ be a discrete topological group, $ \{\chi_g\}_{g\in G}$ be the  standard orthonormal  basis for $\ell^2(G) $.  Let $\lambda $ be the left regular representation of $ G$ defined by $ \lambda_g\chi_q(r)=\chi_q(g^{-1}r), \forall  g, q, r \in G$;  $\rho $ be the right regular representation of $ G$ defined by $ \rho_g\chi_q(r)=\chi_q(rg), \forall g, q, r \in G.$ By $\mathscr{L}(G) $, we mean  the von Neumann algebra generated by unitaries $\{\lambda_g\}_{g\in G} $ in $ \mathcal{B}(\ell^2(G))$. Similarly $\mathscr{R}(G) $ denotes the von Neumann algebra generated by $\{\rho_g\}_{g\in G} $ in $ \mathcal{B}(\ell^2(G))$. We  recall that $\mathscr{L}(G)'=\mathscr{R}(G)$, $ \mathscr{R}(G)'=\mathscr{L}(G) $ \cite{CONWAY}.

\begin{definition}
	Let $ \pi$ be a unitary representation of a discrete 
	group $ G$ on  a Hilbert space $ \mathcal{H}.$ An operator $ A$ in $ \mathcal{B}(\mathcal{H}, \mathcal{H}_0)$ is called a factorable   operator frame generator (resp. a  Parseval frame generator) w.r.t. an operator $ \Psi$ in $ \mathcal{B}(\mathcal{H}, \mathcal{H}_0)$ if $(\{A_g\coloneqq A \pi_{g^{-1}}\}_{g\in G}, \{\Psi_g\coloneqq \Psi \pi_{g^{-1}}\}_{g\in G})$ is a factorable weak OVF (resp.  Parseval)  in $ \mathcal{B}(\mathcal{H}, \mathcal{H}_0)$. In this case, we write $ (A,\Psi)$ is an operator  frame generator for $\pi$.
\end{definition}

\begin{proposition}\label{REPRESENATIONLEMMA}
	Let $ (A,\Psi)$ and $ (B,\Phi)$ be   operator frame generators    in $\mathcal{B}(\mathcal{H},  \mathcal{H}_0)$ for a unitary representation $ \pi$ of  $G$ on $ \mathcal{H}.$ Then
	\begin{enumerate}[\upshape(i)]
		\item $ \theta_A\pi_g=(\lambda_g\otimes I_{\mathcal{H}_0})\theta_A,  \theta_\Psi \pi_g=(\lambda_g\otimes I_{\mathcal{H}_0})\theta_\Psi,  \forall g \in G.$
		\item $ \theta_A^*\theta_B,   \theta_\Psi^*\theta_\Phi,\theta_A^*\theta_\Phi$ are in the commutant $ \pi(G)'$ of $ \pi(G)''.$ Further, $ S_{A,\Psi} \in \pi(G)'$. 
		\item $ \theta_AT\theta_\Psi^*, \theta_AT\theta_B^*, \theta_\Psi T\theta_\Phi^* \in \mathscr{R}(G)\otimes \mathcal{B}(\mathcal{H}_0), \forall T \in \pi(G)'.$ In particular, $ P_{A, \Psi} \in \mathscr{R}(G)\otimes \mathcal{B}(\mathcal{H}_0). $
	\end{enumerate}
\end{proposition}
\begin{proof} Let $ g,p,q \in G $ and $ h \in \mathcal{H}_0.$
	\begin{enumerate}[\upshape(i)]
		\item  From the definition of $ \lambda_g $ and $ \chi_q$,  we get $ \lambda_g\chi_q=\chi_{gq}.$ Therefore $ L_{gq}h=\chi_{gq}\otimes h= \lambda_g\chi_q\otimes h= (\lambda_g\otimes I_{\mathcal{H}_0})(\chi_q\otimes h)=(\lambda_g\otimes I_{\mathcal{H}_0})L_qh.$  Using this, 
		\begin{align*}
			\theta_A\pi_g&=\sum\limits_{p\in G} L_pA_p\pi_g=\sum\limits_{p\in G} L_pA\pi_{p^{-1}}\pi_g=\sum\limits_{p\in G} L_pA\pi_{{p^{-1}}g}\\
			&=\sum\limits_{q\in G} L_{gq}A\pi_{q^{-1}}=\sum\limits_{q\in G}(\lambda_g\otimes I_{\mathcal{H}_0}) L_{q}A\pi_{q^{-1}}=(\lambda_g\otimes I_{\mathcal{H}_0})\theta_A. 
		\end{align*}
		Similarly $ \theta_\Psi \pi_g=(\lambda_g\otimes I_{\mathcal{H}_0})\theta_\Psi.$
		\item $ \theta_A^*\theta_B\pi_g=\theta_A^* (\lambda_g\otimes I_{\mathcal{H}_0})\theta_B=((\lambda_{g^{-1}}\otimes I_{\mathcal{H}_0})\theta_A)^*\theta_B=(\theta_A\pi_{g^{-1}})^*\theta_B=\pi_g\theta_A^*\theta_B.$ In the same way, $ \theta_\Psi^*\theta_\Phi, \theta_A^*\theta_\Phi\in \pi(G)'.$ By taking $ B=A$ and $ \Phi=\Psi$ we get  $ S_{A,\Psi} \in \pi(G)'.$ 
		\item Let  $ T \in \pi(G)'.$ Then 
		$$\theta_AT\theta_\Psi^*(\lambda_g\otimes I_{\mathcal{H}_0})= \theta_AT((\lambda_{g^{-1}}\otimes I_{\mathcal{H}_0})\theta_\Psi)^*=\theta_AT\pi_g\theta_\Psi^*=\theta_A\pi_gT\theta_\Psi^*=(\lambda_g\otimes I_{\mathcal{H}_0})\theta_AT\theta_\Psi^*.$$
		From the construction of $ \mathscr{L}(G),$ we now get $\theta_AT\theta_\Psi^* \in (\mathscr{L}(G)\otimes \{I_{\mathcal{H}_0}\})'=\mathscr{L}(G)'\otimes \{I_{\mathcal{H}_0}\}'=\mathscr{R}(G)\otimes \mathcal{B}(\mathcal{H}_0).$ Similarly $\theta_AT\theta_B^*, \theta_\Psi S\theta_\Phi^* \in \mathscr{R}(G)\otimes \mathcal{B}(\mathcal{H}_0), \forall  S \in \pi (G)'.$ For the choice  $ T=S_{A,\Psi}^{-1}$ we get $ P_{A, \Psi} \in \mathscr{R}(G)\otimes \mathcal{B}(\mathcal{H}_0). $
	\end{enumerate}
\end{proof}

\begin{theorem}\label{gc1}
	Let $ G$ be a discrete group,  $ e$ be the identity of $G$ and $( \{A_g\}_{g\in G},  \{\Psi_g\}_{g\in G})$ be a Parseval  factorable weak OVF  in $ \mathcal{B}(\mathcal{H},\mathcal{H}_0).$ Then there is a  unitary representation $ \pi$  of $ G$ on  $ \mathcal{H}$  for which 
	$$ A_g=A_e\pi_{g^{-1}}, ~\Psi_g=\Psi_e\pi_{g^{-1}}, ~\forall  g \in G$$
	if and only if 
	$$A_{gp}A_{gq}^*=A_pA_q^* ,~ A_{gp}\Psi_{gq}^*=A_p\Psi_q^*,~ \Psi_{gp}\Psi_{gq}^*=\Psi_p\Psi_q^*, \quad  \forall g,p,q \in G.$$
\end{theorem} 
\begin{proof}
	$(\Rightarrow)$
	$$A_{gp}\Psi_{gq}^*= A_e \pi_{(gp)^{-1}}(\Psi_e\pi_{(gq)^{-1}})^*=A_e\pi_{p^{-1}}\pi_{g^{-1}}\pi_g\pi_q\Psi_e^*=A_p\Psi_q^*, \quad\forall g,p,q \in G.$$
	Similarly we get other two equalities.
	
	 $(\Leftarrow)$ Using  assumptions, we use the following three equalities in the proof, among them  we derive the second, remainings are similar.
	For all $ g \in G,$
	\begin{align*}
	&	(\lambda_g\otimes I_{\mathcal{H}_0})\theta_A\theta_A^*=\theta_A\theta_A^*(\lambda_g\otimes I_{\mathcal{H}_0}), ~ (\lambda_g\otimes I_{\mathcal{H}_0})\theta_A\theta_\Psi^*=\theta_A\theta_\Psi^*(\lambda_g\otimes I_{\mathcal{H}_0}),\\
	&	(\lambda_g\otimes I_{\mathcal{H}_0})\theta_\Psi\theta_\Psi^*=\theta_\Psi\theta_\Psi^*(\lambda_g\otimes I_{\mathcal{H}_0}).
	\end{align*}
	Noticing $ \lambda_g$ is unitary, we get  $(\lambda_g\otimes I_{\mathcal{H}_0})^{-1}=(\lambda_g\otimes I_{\mathcal{H}_0})^*$; also we observed in the proof of Proposition \ref{REPRESENATIONLEMMA} that  $(\lambda_g\otimes I_{\mathcal{H}_0})L_q=L_{gq}.$ So
	\begin{align*}
		(\lambda_g\otimes I_{\mathcal{H}_0})\theta_A\theta_\Psi^*(\lambda_g\otimes I_{\mathcal{H}_0})^*&=\left(\sum\limits_{p\in G}(\lambda_g\otimes I_{\mathcal{H}_0})L_pA_p\right)\left(\sum\limits_{q\in G}(\lambda_g\otimes I_{\mathcal{H}_0})L_q\Psi_q\right)^*\\
		&=\sum\limits_{p\in G} L_{gp}\left(\sum\limits_{q\in G}A_p\Psi_q^*L_{gq}^*\right)
		=\sum\limits_{r\in G} L_r\left(\sum\limits_{s\in G}A_{g^{-1}r}\Psi_{g^{-1}s}^*L_s^*\right)\\
		& =\sum\limits_{r\in G} L_r\left(\sum\limits_{s\in G}A_r\Psi_s^*L_s^*\right)=\theta_A\theta_\Psi^*.
	\end{align*}
	Define $ \pi : G \ni g  \mapsto \pi_g\coloneqq \theta_\Psi^*(\lambda_g\otimes I_{\mathcal{H}_0})\theta_A  \in \mathcal{B}(\mathcal{H}).$ By using the  Parsevalness, 
	\begin{align*}
	\pi_g\pi_h&=\theta_\Psi^*(\lambda_g\otimes I_{\mathcal{H}_0})\theta_A \theta_\Psi^*(\lambda_h\otimes I_{\mathcal{H}_0})\theta_A =\theta_\Psi^*\theta_A \theta_\Psi^*(\lambda_g\otimes I_{\mathcal{H}_0}) (\lambda_h\otimes I_{\mathcal{H}_0})\theta_A \\
	&= \theta_\Psi^*(\lambda_{gh}\otimes I_{\mathcal{H}_0})\theta_A =\pi_{gh}, \quad \forall g, h \in G
	\end{align*}
	 and 
	 \begin{align*}
	 \pi_g\pi_g^*&=\theta_\Psi^*(\lambda_g\otimes I_{\mathcal{H}_0})\theta_A\theta_A^*(\lambda_{g^{-1}}\otimes I_{\mathcal{H}_0})\theta_\Psi\\
	 &=\theta_\Psi^*\theta_A\theta_A^*(\lambda_g\otimes I_{\mathcal{H}_0})(\lambda_{g^{-1}}\otimes I_{\mathcal{H}_0})\theta_\Psi=I_\mathcal{H}, \\ \pi_g^*\pi_g&=\theta_A^*(\lambda_{g^{-1}}\otimes I_{\mathcal{H}_0})\theta_\Psi\theta_\Psi^*(\lambda_{g}\otimes I_{\mathcal{H}_0})\theta_A\\
	 &=\theta_A^*(\lambda_{g^{-1}}\otimes I_{\mathcal{H}_0})(\lambda_{g}\otimes I_{\mathcal{H}_0})\theta_\Psi\theta_\Psi^*\theta_A=I_\mathcal{H},   \quad \forall  g \in G.
	 \end{align*}
	   Since $ G $ has the discrete topology, this proves $ \pi$ is a unitary representation. It remains to prove  $ A_g=A_e\pi_{g^{-1}}, \Psi_g=\Psi_e\pi_{g^{-1}}  $ for all $ g \in G$. Indeed,
	$$A_e\pi_{g^{-1}}= L_e^*\theta_A\theta_\Psi^*(\lambda_{g^{-1}}\otimes I_{\mathcal{H}_0})\theta_A=L_e^*(\lambda_{g^{-1}}\otimes I_{\mathcal{H}_0})\theta_A\theta_\Psi^*\theta_A=((\lambda_g\otimes I_{\mathcal{H}_0})L_e)^*\theta_A=L_{ge}^*\theta_A=A_g,$$ 
	and
	$$\Psi_e\pi_{g^{-1}}=L_e^*\theta_\Psi \theta_\Psi^*(\lambda_{g^{-1}}\otimes I_{\mathcal{H}_0})\theta_A=L_e^*(\lambda_{g^{-1}}\otimes I_{\mathcal{H}_0})\theta_\Psi\theta_\Psi^*\theta_A=((\lambda_g\otimes I_{\mathcal{H}_0})L_e)^*\theta_\Psi=L_{ge}^*\theta_\Psi=\Psi_g.$$
\end{proof}
In the direct part of Theorem \ref{gc1},   we can remove the word  `Parseval' since it has not been used in the proof;  same is true in the following corollary.

\begin{corollary}
	Let $ G$ be a discrete group, $e$ be the identity of $G$ and $( \{A_g\}_{g\in G},  \{\Psi_g\}_{g\in G})$ be a factorable weak OVF  in $ \mathcal{B}(\mathcal{H},\mathcal{H}_0).$ Then there is a  unitary representation $ \pi$  of $ G$ on  $ \mathcal{H}$  for which
	\begin{enumerate}[\upshape(i)]
		\item  $ A_g=A_eS_{A,\Psi}^{-1}\pi_{g^{-1}}S_{A,\Psi}, \Psi_g=\Psi_e\pi_{g^{-1}}  $ for all $ g \in G$  if and only if $$A_{gp}S_{A,\Psi}^{-1}(S_{A,\Psi}^{-1})^*A_{gq}^*=A_pS_{A,\Psi}^{-1}(S_{A,\Psi}^{-1})^*A_q^* ,\quad A_{gp}S_{A,\Psi}^{-1}\Psi_{gq}^*=A_pS_{A,\Psi}^{-1}\Psi_q^*, \quad \Psi_{gp}\Psi_{gq}^*=\Psi_p\Psi_q^*, \quad  \forall  g,p,q \in G.$$
	\item  $ A_g=A_e\pi_{g^{-1}}, \Psi_g=\Psi_e(S_{A,\Psi}^{-1})^*\pi_{g^{-1}}S_{A,\Psi}  $ for all $ g \in G$  if and only if 
	$$A_{gp}A_{gq}^*=A_pA_q^* ,\quad A_{gp}S_{A,\Psi}^{-1}\Psi_{gq}^*=A_pS_{A,\Psi}^{-1}\Psi_q^*, \quad  \Psi_{gp}(S_{A,\Psi}^{-1})^*S_{A,\Psi}^{-1}\Psi_{gq}^*=\Psi_p(S_{A,\Psi}^{-1})^*S_{A,\Psi}^{-1}\Psi_q^*, \quad \forall  g,p,q \in G.$$
	\end{enumerate}	
\end{corollary}
\begin{proof}
	We apply Theorem  \ref{gc1} to the factorable Parseval OVF 
	\begin{enumerate}[\upshape(i)]
		\item  $(\{A_gS_{A,\Psi}^{-1}\}_{g\in G}, \{\Psi_g\}_{g\in G})$ to get: there is a  unitary representation $ \pi$  of $ G$ on  $ \mathcal{H}$  for which $ A_gS_{A,\Psi}^{-1}=(A_eS_{A,\Psi}^{-1})\pi_{g^{-1}}, \Psi_g=\Psi_e\pi_{g^{-1}}  $ for all $ g \in G$  if and only if 
		\begin{align*}
		&(A_{gp}S_{A,\Psi}^{-1})(A_{gq}S_{A,\Psi}^{-1})^*=(A_pS_{A,\Psi}^{-1})(A_qS_{A,\Psi}^{-1})^*, \quad (A_{gp}S_{A,\Psi}^{-1})\Psi_{gq}^*= (A_pS_{A,\Psi}^{-1})\Psi_q^*,\\
		&\Psi_{gp}\Psi_{gq}^*=\Psi_p\Psi_q^*, \quad \forall  g,p,q \in G.
		\end{align*}
	\item $( \{A_g\}_{g\in G}, \{\Psi_g(S_{A,\Psi}^{-1})^*\}_{g\in G})$ to get: there is a  unitary representation $ \pi$  of $ G$ on  $ \mathcal{H}$  for which $ A_g=A_e\pi_{g^{-1}}, \Psi_gS_{A,\Psi}^{-1}=(\Psi_e(S_{A,\Psi}^{-1})^*)\pi_{g^{-1}}  $ for all $ g \in G$  if and only if 
	\begin{align*}
	&A_{gp}A_{gq}^*=A_pA_q^* , \quad A_{gp}(\Psi_{gq}(S_{A,\Psi}^{-1})^*)^*=A_p(\Psi_q(S_{A,\Psi}^{-1})^*)^*,\\
	&(\Psi_{gp}(S_{A,\Psi}^{-1})^*)(\Psi_{gq}(S_{A,\Psi}^{-1})^*)^*=(\Psi_p(S_{A,\Psi}^{-1})^*)(\Psi_q(S_{A,\Psi}^{-1})^*)^* , \quad \forall  g,p,q \in G.
	\end{align*}
	\end{enumerate}		
\end{proof}

\section{Frames and group-like unitary systems} \label{FRAMESANDGROUP-LIKEUNITARYSYSTEMS}
We next address the situation of factorable weak OVF whenever it is indexed by group-like unitary systems. Group-like unitary systems arose from the study of Weyl-Heisenberg frames. To the best of our knowledge, this was first formally defined by Gabardo and Han in 2001 \cite{GABARDO}. In the sequel, by $\mathbb{T}$, we mean the standard unit circle group centred at the origin equipped with usual multiplication.
\begin{definition}\cite{GABARDO}
	A collection   $ \mathcal{U}\subseteq \mathcal{B}(\mathcal{H})$  containing $I_\mathcal{H}$ is called as a unitary system.  If the group generated by  unitary system $ \mathcal{U}$, denoted by $ \operatorname{group}(\mathcal{U})$ is such that 
	\begin{enumerate}[\upshape(i)]
		\item $\operatorname{group}(\mathcal{U}) \subseteq \mathbb{T}\mathcal{U}\coloneqq \{\alpha U : \alpha \in \mathbb{T}, U\in \mathcal{U}  \}$, and 
		\item $\mathcal{U}$ is linearly independent, i.e.,  $\mathbb{T}U\ne\mathbb{T}V $ whenever $ U, V \in \mathcal{U}$ are such that $ U\ne V,$ 
	\end{enumerate}
	then $\mathcal{U}$ is called as a group-like unitary system.
\end{definition}

Let $ \mathcal{U}$ be a group-like unitary system. As in \cite{GABARDOHANGROUPLIKE}, we  define mappings 
\begin{align*}
f:\operatorname{group}(\mathcal{U})\rightarrow \mathbb{T} \quad \text{ and } \quad \sigma:\operatorname{group}(\mathcal{U})\rightarrow \mathcal{U}.
\end{align*}
 in the following way. For each  $ U \in  \operatorname{group}(\mathcal{U}) $ there are unique $\alpha\in \mathbb{T}, V \in \mathcal{U} $  such that $ U=\alpha V$. Define $ f(U)=\alpha$ and $\sigma(U)=V $. These  $ f, \sigma $ are well-defined and satisfy 
\begin{align*}
 U=f(U)\sigma(U), \quad \forall U \in \operatorname{group}(\mathcal{U}).
\end{align*}
 These mappings are called as \textit{corresponding mappings} associated to $ \mathcal{U}$. We can picturize these maps as follows. 
 	\begin{center}
 	\[
 	\begin{tikzcd}
 \operatorname{group}(\mathcal{U}) \arrow[d,"\sigma"] \arrow[dr,"f"]\subseteq\mathbb{T}\mathcal{U}\\
 	\mathcal{U}  & \mathbb{T}\arrow[l,] \\
 	\end{tikzcd}
 	\]
 \end{center}
Next result gives certain fundamental properties of corresponding mappings associated with group-like unitary systems.
\begin{proposition}\cite{GABARDOHANGROUPLIKE}\label{PER} 
	For a group-like unitary system $\mathcal{U}$ and $ f, \sigma $ as above,
	\begin{enumerate}[\upshape(i)]
		\item $ f(U\sigma(VW))f(VW)=f(\sigma(UV)W)f(UV), \forall U,V,W \in \operatorname{group}(\mathcal{U}).$
		\item $ \sigma(U\sigma(VW))=\sigma(\sigma(UV)W), \forall U,V,W \in \operatorname{group} (\mathcal{U}).$
		\item $ \sigma(U)=U$ and $ f(U)=1$ for all $ U \in \mathcal{U}.$
		\item If $  V, W \in \operatorname{group} (\mathcal{U}),$ then
		\begin{align*}
			\mathcal{U}&=\{\sigma(UV) : U \in \mathcal{U}\}=\{\sigma(VU^{-1}) : U \in \mathcal{U}\}\\
			&=\{\sigma(VU^{-1}W) : U \in \mathcal{U}\}=\{\sigma(V^{-1}U) : U \in \mathcal{U}\}.
		\end{align*}
		\item For fixed  $  V, W \in \mathcal{U}$, the following mappings are  injective from  $ \mathcal{U} $ to itself:
		$$ U\mapsto \sigma(VU) \quad (\text{resp.} ~ \sigma(UV), \sigma(UV^{-1}), \sigma(V^{-1}U), \sigma(VU^{-1}), \sigma(U^{-1}V), \sigma(VU^{-1}W)).$$
	\end{enumerate}
\end{proposition}
Since $\operatorname{group} (\mathcal{U}) $ is a group, we note that, in (iv) of Proposition \ref{PER},  we can replace $V$ by $V^{-1}$. Hence, whenever  $V \in \operatorname{group} (\mathcal{U})$, we have $\sum_{U \in \mathcal{U}}x_U=\sum_{U \in \mathcal{U}}x_{\sigma(VU)}$.

\begin{definition}\cite{GABARDOHANGROUPLIKE}
	A unitary representation $ \pi$ of a group-like unitary system $   \mathcal{U}$ on $ \mathcal{H}$ is an injective mapping from $  \mathcal{U}$ into the set of unitary operators on $ \mathcal{H}$ such that 
	$$\pi(U)\pi(V)=f(UV)\pi(\sigma(UV)) , \quad {\pi(U)}^{-1}=f(U^{-1})\pi(\sigma(U^{-1})), ~ \forall U,V \in \mathcal{U}, $$
	where $ f$ and $ \sigma $  are the corresponding mappings associated with $  \mathcal{U}.$ 
\end{definition}
Since $\pi $ is injective, once we have a unitary representation of a group-like unitary system $   \mathcal{U}$ on $\mathcal{H}$, then $ \pi(\mathcal{U})$ is also a group-like unitary system.

Let $ \mathcal{U}$ be a  group-like unitary system and  $ \{\chi_U\}_{U\in \mathcal{U}}$ be the  standard orthonormal  basis for $\ell^2(\mathcal{U}) $.  We define $\lambda $ on  $ \mathcal{U}$  by $ \lambda_U\chi_V=f(UV)\chi_{\sigma(UV)}, \forall   U,V \in \mathcal{U}.$ Then $ \lambda $ is a unitary  representation which we call as  left  regular representation of $ \mathcal{U}$. Similarly, we define right regular representation of $ \mathcal{U}$ by $ \rho_U\chi_V=f(VU^{-1})\chi_{\sigma(VU^{-1})}, \forall U,V \in \mathcal{U}$ \cite{GABARDOHANGROUPLIKE}.  
Like frame generators for groups, we now define the frame generator for group-like unitary systems.
\begin{definition}
	Let $ \mathcal{U}$  be a group-like unitary system.  An operator $ A$ in $ \mathcal{B}(\mathcal{H}, \mathcal{H}_0)$ is called an  operator frame generator (resp. a  Parseval  frame generator) w.r.t. $ \Psi$ in $ \mathcal{B}(\mathcal{H}, \mathcal{H}_0)$  if $(\{A_U\coloneqq A\pi(U)^{-1}\}_{U\in \mathcal{U}},\{\Psi_U\coloneqq \Psi\pi(U)^{-1}\}_{U\in \mathcal{U}})$ is a factorable weak OVF  (resp. a Parseval) in $ \mathcal{B}(\mathcal{H}, \mathcal{H}_0)$.  We write $ (A,\Psi)$ is an operator frame generator for $\pi$.
\end{definition}

\begin{theorem}\label{CHARACTERIZATIONGROUPLIKE}
	Let $ \mathcal{U}$ be a  group-like unitary system, $ I$ be the identity of $ \mathcal{U}$ and $(\{A_U\}_{U\in \mathcal{U}},\{\Psi_U\}_{U\in \mathcal{U}})$ be a factorable Parseval weak OVF  in $ \mathcal{B}(\mathcal{H},\mathcal{H}_0)$ with $ \theta_A^*$ injective. Then there is a unitary representation $ \pi$  of $ \mathcal{U}$ on  $ \mathcal{H}$  for which 
	$$ A_U=A_I\pi(U)^{-1}, ~\Psi_U=\Psi_I\pi(U)^{-1},~\forall  U \in \mathcal{U}$$
	if and only if 
	\begin{align*}
		A_{\sigma(UV)}A_{\sigma(UW)}^*&=f(UV)\overline{f(UW)} A_VA_W^* ,\\
		A_{\sigma(UV)}\Psi_{\sigma(UW)}^*&=f(UV)\overline{f(UW)} A_V\Psi_W^*,\\ \Psi_{\sigma(UV)}\Psi_{\sigma(UW)}^*&=f(UV)\overline{f(UW)} \Psi_V\Psi_W^*, \quad \forall  U,V,W \in \mathcal{U}.
	\end{align*}
\end{theorem} 

\begin{proof}
	$(\Rightarrow)$ For all $U,V,W \in \mathcal{U}$, we have
	\begin{align*}
		A_{\sigma(UV)}A_{\sigma(UW)}^*&= A_I\pi(\sigma(UV))^{-1}( A_I\pi(\sigma(UW))^{-1} )^*\\
		&=A_I(\overline{f(UV)}\pi(U)\pi(V))^{-1} \overline{f(UW)}\pi(U)\pi(W)A^*_I\\
		&=f(UV)\overline{f(UW)} A_I\pi(V)^{-1}(A_I\pi(W)^{-1})^*\\
		&=f(UV)\overline{f(UW)} A_VA_W^*. 
	\end{align*}
	Others can be shown similarly. 
	
	$(\Leftarrow)$  We have to construct unitary representation which satisfies the stated conditions. Following observation plays an important role in this part. Let $ h\in \mathcal{H}.$ Then $ L_{\sigma(UV)}h=\chi_{\sigma(UV)}\otimes h=\overline{f(UV)}\lambda_U\chi_V\otimes h=\overline{f(UV)}(\lambda_U\chi_V\otimes h)=\overline{f(UV)}(\lambda_U\otimes I_{\mathcal{H}_0})(\chi_V\otimes h)=\overline{f(UV)}(\lambda_U\otimes I_{\mathcal{H}_0})L_V h.$
	
	As in the proof of Theorem \ref{gc1},  we argue the following, for which now we prove the first.
	For all $ U \in \mathcal{U},$
	\begin{align*}
		&(\lambda_U\otimes I_{\mathcal{H}_0})\theta_A\theta_A^*=\theta_A\theta_A^*(\lambda_U\otimes I_{\mathcal{H}_0}), ~ (\lambda_U\otimes I_{\mathcal{H}_0})\theta_A\theta_\Psi^*=\theta_A\theta_\Psi^*(\lambda_U\otimes I_{\mathcal{H}_0}),\\
		&(\lambda_U\otimes I_{\mathcal{H}_0})\theta_\Psi\theta_\Psi^*=\theta_\Psi\theta_\Psi^*(\lambda_U\otimes I_{\mathcal{H}_0}).
	\end{align*}
	Consider 
	\begin{align*}
		(\lambda_U\otimes I_{\mathcal{H}_0})\theta_A\theta_A^*(\lambda_U\otimes I_{\mathcal{H}_0})^*&=\left(\sum\limits_{V\in \mathcal{U}}(\lambda_U\otimes I_{\mathcal{H}_0})L_VA_V\right)\left(\sum\limits_{W\in \mathcal{U}}(\lambda_U\otimes I_{\mathcal{H}_0})L_WA_W\right)^*\\ &=\left(\sum\limits_{V\in \mathcal{U}} f(UV)L_{\sigma(UV)}A_V\right)\left(\sum\limits_{W\in \mathcal{U}}f(UW)L_{\sigma(UW)}A_W\right)^*\\
		&=\sum\limits_{V\in \mathcal{U}} L_{\sigma(UV)}\left(\sum\limits_{W\in \mathcal{U}}f(UV)\overline{f(UW)}A_VA_W^*L_{\sigma(UW)}^*\right)\\
		&= \sum\limits_{V\in \mathcal{U}} L_{\sigma(UV)}\left(\sum\limits_{W\in \mathcal{U}}A_{\sigma(UV)}A_{\sigma(UW)}^*L_{\sigma(UW)}^*\right)\\
		&=\left(\sum\limits_{V\in \mathcal{U}} L_{\sigma(UV)}A_{\sigma(UV)}\right)\left(\sum\limits_{W\in \mathcal{U}}L_{\sigma(UW)}A_{\sigma(UW)}\right)^*\\
		&=\theta_A\theta_A^*
	\end{align*}
	where last part of Proposition \ref{PER} is used in the last equality.
	
	Define $ \pi : \mathcal{U} \ni U  \mapsto \pi(U)\coloneqq \theta_\Psi^*(\lambda_U\otimes I_{\mathcal{H}_0})\theta_A  \in \mathcal{B}(\mathcal{H}).$  Then $ \pi(U)\pi(V)=\theta_\Psi^*(\lambda_U\otimes I_{\mathcal{H}_0})\theta_A \theta_\Psi^*(\lambda_V\otimes I_{\mathcal{H}_0})\theta_A =\theta_\Psi^*\theta_A \theta_\Psi^*(\lambda_U\otimes I_{\mathcal{H}_0}) (\lambda_V\otimes I_{\mathcal{H}_0})\theta_A = \theta_\Psi^*(\lambda_U\lambda_V\otimes I_{\mathcal{H}_0})\theta_A =\theta_\Psi^*(f(UV)\lambda_{\sigma(UV)}\otimes I_{\mathcal{H}_0})\theta_A  =f(UV) \theta_\Psi^*(\lambda_{\sigma(UV)}\otimes I_{\mathcal{H}_0})\theta_A =f(UV)\pi({\sigma(UV)})$ for all $U, V \in \mathcal{U},$ and $\pi(U)\pi(U)^*=\theta_\Psi^*(\lambda_U\otimes I_{\mathcal{H}_0})\theta_A\theta_A^*(\lambda_U^*\otimes I_{\mathcal{H}_0})\theta_\Psi=\theta_\Psi^*\theta_A\theta_A^*(\lambda_U\otimes I_{\mathcal{H}_0})(\lambda_U^*\otimes I_{\mathcal{H}_0})\theta_\Psi=I_\mathcal{H},  \pi(U)^*\pi(U)=\theta_A^*(\lambda_U^*\otimes I_{\mathcal{H}_0})\theta_\Psi\theta_\Psi^*(\lambda_{U}\otimes I_{\mathcal{H}_0})\theta_A=\theta_A^*(\lambda_U^*\otimes I_{\mathcal{H}_0})(\lambda_U\otimes I_{\mathcal{H}_0})\theta_\Psi\theta_\Psi^*\theta_A=I_\mathcal{H} $ for all $ U \in \mathcal{U}$. Further, 
	\begin{align*}
		\pi(U)f(U^{-1})\pi(\sigma(U^{-1}))&=\theta_\Psi^*(\lambda_U\otimes I_{\mathcal{H}_0})\theta_Af(U^{-1})\theta_\Psi^*(\lambda_{\sigma(U^{-1})}\otimes I_{\mathcal{H}_0})\theta_A \\
		&=f(U^{-1})\theta_\Psi^*\theta_A\theta_\Psi^*(\lambda_U\otimes I_{\mathcal{H}_0})(\lambda_{\sigma(U^{-1})}\otimes I_{\mathcal{H}_0})\theta_A \\
		&=f(U^{-1})\theta_\Psi^*(\lambda_U\otimes I_{\mathcal{H}_0})(\lambda_{\sigma(U^{-1})}\otimes I_{\mathcal{H}_0})\theta_A\\
		&=f(U^{-1})\theta_\Psi^*(\lambda_U\lambda_{\sigma(U^{-1})}\otimes I_{\mathcal{H}_0})\theta_A\\
		&=f(U^{-1})\theta_\Psi^*(f(U\sigma(U^{-1}))\lambda_{\sigma(U\sigma (U^{-1}))}\otimes I_{\mathcal{H}_0})\theta_A\\
		&=\theta_\Psi^*(f(U\sigma(U^{-1}I))f(U^{-1}I)\lambda_{\sigma(U\sigma (U^{-1}I))}\otimes I_{\mathcal{H}_0})\theta_A\\
		&=\theta_\Psi^*(f(\sigma(UU^{-1})I)f(UU^{-1})\lambda_{\sigma({\sigma(UU^{-1})I})}\otimes I_{\mathcal{H}_0})\theta_A\\
		&=\theta_\Psi^*(\lambda_I\otimes I_{\mathcal{H}_0})\theta_A=I_\mathcal{H}
	\end{align*}
	$\Rightarrow {\pi(U)}^{-1}=f(U^{-1})\pi(\sigma(U^{-1}))$ for all $ U \in \mathcal{U}$. We shall now use $ \theta_A^*$ is injective   to show $ \pi$ is injective and thereby to  get $ \pi$ is a unitary representation. Let $ \pi(U)=\pi(V).$ Then 
	\begin{align*}
	&\theta_\Psi^*(\lambda_U\otimes I_{\mathcal{H}_0})\theta_A =\theta_\Psi^*(\lambda_V\otimes I_{\mathcal{H}_0})\theta_A  \Rightarrow \theta_\Psi^*(\lambda_U\otimes I_{\mathcal{H}_0})\theta_A \theta_A^*=\theta_\Psi^*(\lambda_V\otimes I_{\mathcal{H}_0})\theta_A  \theta_A^* \\
	&\Rightarrow \theta_\Psi^*\theta_A  \theta_A^*(\lambda_U\otimes I_{\mathcal{H}_0}) =\theta_\Psi^*\theta_A  \theta_A^*(\lambda_V\otimes I_{\mathcal{H}_0}) \Rightarrow \lambda_U\otimes I_{\mathcal{H}_0}=\lambda_V\otimes I_{\mathcal{H}_0}.
	\end{align*}
	 We show $ U$ and $ V$ are identical at elementary tensors. For $ h \in \ell^2(\mathcal{U}),  y \in \mathcal{H}_0,  $ we get, $(\lambda_U\otimes I_{\mathcal{H}_0})(h\otimes y)=(\lambda_V\otimes I_{\mathcal{H}_0})(h\otimes y)\Rightarrow \lambda_Uh\otimes y=\lambda_Vh\otimes y \Rightarrow (\lambda_U-\lambda_V)h\otimes y=0 \Rightarrow 0= \langle  (\lambda_U-\lambda_V)h\otimes y, (\lambda_U-\lambda_V)h\otimes y\rangle= \|(\lambda_U-\lambda_V)h\|^2 \|y\|^2 .$ We may assume $y\neq0$ (if  $y=0$, then  $h\otimes y=0$). But then $ (\lambda_U-\lambda_V)(h)=0,$ and $ \lambda$ is a unitary representation (it is injective) gives $ U=V.$ The pending part  $ A_U=A_I\pi(U)^{-1}, \Psi_U=\Psi_I\pi(U)^{-1}  $ for all $ U \in \mathcal{U}$ we show, now.
	\begin{align*}
		A_I\pi(U)^{-1}&= L_I^*\theta_A(\theta_\Psi^*(\lambda_U\otimes I_{\mathcal{H}_0})\theta_A)^*
		= L_I^*(\theta_\Psi^*(\lambda_U\otimes I_{\mathcal{H}_0})\theta_A\theta_A^*)^*\\
		&=L_I^*(\theta_\Psi^*\theta_A\theta_A^*(\lambda_U\otimes I_{\mathcal{H}_0}))^*
		 = L_I^*(\theta_A^*(\lambda_U\otimes I_{\mathcal{H}_0}))^*\\
		 &=(\theta_A^*(\lambda_U\otimes I_{\mathcal{H}_0})L_I)^*= (\theta_A^*\overline{f(UI)}(\lambda_U\otimes I_{\mathcal{H}_0})L_I)^*\\
		&= (\theta_A^*L_{\sigma({UI})})^*=L_U^*\theta_A=A_U
	\end{align*}
	and
	\begin{align*}
		\Psi_I\pi(U)^{-1}&= L_I^*\theta_\Psi(\theta_\Psi^*(\lambda_U\otimes I_{\mathcal{H}_0})\theta_A)^*
		= L_I^*(\theta_\Psi^*(\lambda_U\otimes I_{\mathcal{H}_0})\theta_A\theta_\Psi^*)^*\\
		&=L_I^*(\theta_\Psi^*\theta_A\theta_\Psi^*(\lambda_U\otimes I_{\mathcal{H}_0}))^*
		 = L_I^*(\theta_\Psi^*(\lambda_U\otimes I_{\mathcal{H}_0}))^*\\
		 &=(\theta_\Psi^*(\lambda_U\otimes I_{\mathcal{H}_0})L_I)^*= (\theta_\Psi^*\overline{f(UI)}(\lambda_U\otimes I_{\mathcal{H}_0})L_I)^*\\
		&= (\theta_\Psi^*L_{\sigma({UI})})^*=L_U^*\theta_\Psi=\Psi_U.
	\end{align*}
\end{proof}
Note that neither  Parsevalness of the frame nor $ \theta_A^*$  is injective was used  in the direct part of Theorem \ref{CHARACTERIZATIONGROUPLIKE}. Since $ \theta_A$ acts between Hilbert spaces, we know that $ \overline{\theta_A(\mathcal{H})}=\operatorname{Ker}(\theta_A^*)^\perp$ and $ \operatorname{Ker}(\theta_A^*)=\theta_A(\mathcal{H})^\perp.$ From Lemma \ref{DILATIONLEMMA}, the range of $\theta_A$  is closed. Therefore $ \theta_A(\mathcal{H})=\operatorname{Ker}(\theta_A^*)^\perp.$ Thus the condition $ \theta_A^*$ is injective in the Theorem \ref{CHARACTERIZATIONGROUPLIKE} can be replaced by $ \theta_A$ is onto. 
\begin{corollary}
	Let $ \mathcal{U}$ be a  group-like unitary system,  $ I$ be the identity of $ \mathcal{U}$ and $(\{A_U\}_{U\in \mathcal{U}},\{\Psi_U\}_{U\in \mathcal{U}})$ be a factorable weak OVF   in $ \mathcal{B}(\mathcal{H},\mathcal{H}_0)$  with $ \theta_A^*$   is injective. Then there is a unitary representation $ \pi$  of  $ \mathcal{U}$ on  $ \mathcal{H}$  for which
	\begin{enumerate}[\upshape(i)]
		\item    $ A_U=A_IS^{-1}_{A,\Psi}\pi(U)^{-1}S_{A, \Psi}, \Psi_U=\Psi_I\pi(U)^{-1}  $ for all $ U \in \mathcal{U}$  if and only if 
		\begin{align*}
		&A_{\sigma(UV)}S^{-1}_{A, \Psi}(S_{A,\Psi}^{-1})^*A_{\sigma(UW)}^*=f(UV)\overline{f(UW)} A_VS^{-1}_{A, \Psi}(S_{A,\Psi}^{-1})^*A_W^*,\\
	&	A_{\sigma(UV)}S_{A, \Psi}^{-1}\Psi_{\sigma(UW)}^*=f(UV)\overline{f(UW)} A_VS^{-1}_{A, \Psi}\Psi_W^*,\\
		&\Psi_{\sigma(UV)}\Psi_{\sigma(UW)}^*=f(UV)\overline{f(UW)} \Psi_V\Psi_W^*, \quad \forall U,V,W \in \mathcal{U}.
		\end{align*}
	\item   $ A_U=A_I\pi(U)^{-1}, \Psi_U =\Psi_I(S_{A,\Psi}^{-1})^*\pi(U)^{-1}S_{A, \Psi}  $ for all $ U \in \mathcal{U}$  if and only if 
	\begin{align*}
	&A_{\sigma(UV)}A_{\sigma(UW)}^*=f(UV)\overline{f(UW)} A_VA_W^*,\\
	&A_{\sigma(UV)}S^{-1}_{A, \Psi}\Psi_{\sigma(UW)}^*=f(UV)\overline{f(UW)}A_VS^{-1}_{A, \Psi}\Psi_W^*,\\
	&\Psi_{\sigma(UV)}(S_{A,\Psi}^{-1})^*S^{-1}_{A, \Psi}\Psi_{\sigma(UW)} ^*
	=f(UV)\overline{f(UW)} \Psi_V(S_{A,\Psi}^{-1})^*S^{-1}_{A, \Psi}\Psi_W^*, \quad \forall U,V,W \in \mathcal{U}.
	\end{align*}
	\end{enumerate}
\end{corollary}
\begin{proof}
	We try to apply Theorem \ref{CHARACTERIZATIONGROUPLIKE} to the factorable  Parseval OVF 
	\begin{enumerate}[\upshape(i)]
		\item  $(\{A_US_{A,\Psi}^{-1}\}_{U\in \mathcal{U}} , \{\Psi_U\}_{U\in \mathcal{U}})$. Then  there is a unitary representation $ \pi$  of  $ \mathcal{U}$ on  $ \mathcal{H}$  for which  $ A_US_{A,\Psi}^{-1}=(A_IS^{-1}_{A,\Psi})\pi(U)^{-1}, \Psi_U=\Psi_I\pi(U)^{-1}  $ for all $ U \in \mathcal{U}$  if and only if 
		\begin{align*}
		&(A_{\sigma(UV)}S^{-1}_{A, \Psi})(A_{\sigma(UW)}S_{A,\Psi}^{-1})^*=f(UV)\overline{f(UW)}( A_VS^{-1}_{A, \Psi})(A_WS_{A,\Psi}^{-1})^*,\\
		&(A_{\sigma(UV)}S_{A, \Psi}^{-1}) \Psi_{\sigma(UW)}^*=
		f(UV)\overline{f(UW)}( A_VS^{-1}_{A, \Psi})\Psi_W^*,\\
		&\Psi_{\sigma(UV)}\Psi_{\sigma(UW)}^*=f(UV)\overline{f(UW)} \Psi_V\Psi_W^*, \quad \forall U,V,W \in \mathcal{U}.
		\end{align*}
		\item  $(\{A_U\}_{U\in \mathcal{U}} , \{\Psi_U(S_{A,\Psi}^{-1})^*\}_{U\in \mathcal{U}})$. Then there is a unitary representation $ \pi$  of  $ \mathcal{U}$ on  $ \mathcal{H}$  for which  $ A_U=A_I\pi(U)^{-1}, \Psi_U(S_{A,\Psi}^{-1})^*=(\Psi_I(S_{A,\Psi}^{-1})^*)\pi(U)^{-1}  $ for all $ U \in \mathcal{U}$  if and only if 
		\begin{align*}
		&A_{\sigma(UV)}A_{\sigma(UW)}^*=
		f(UV)\overline{f(UW)} A_VA_W^*,\\
		&A_{\sigma(UV)}(\Psi_{\sigma(UW)}(S_{A,\Psi}^{-1})^*)^*=f(UV)\overline{f(UW)}A_V(\Psi_W(S_{A,\Psi}^{-1})^*)^*,\\
		&(\Psi_{\sigma(UV)}(S_{A,\Psi}^{-1})^*)(\Psi_{\sigma(UW)}(S_{A,\Psi}^{-1})^*)^*=f(UV)\overline{f(UW)} (\Psi_V(S_{A,\Psi}^{-1})^*)(\Psi_W(S_{A,\Psi}^{-1})^*)^*, \quad \forall U,V,W \in \mathcal{U}.
		\end{align*}
	\end{enumerate}
\end{proof}

\section{Perturbations}\label{PERTURBATIONS}

In this last section we study the stability of factorable weak OVFs. Recall that Paley-Wiener theorem \cite{ARSOVE} in Banach spaces says that sequences which are close to Schauder basis are again Schauder bases. Since a frame will also give a series representation,  it is natural to ask whether a sequence close to frame is a frame. This was first derived by Christensen in 1995 \cite{PALEY1}. Three months later, Christensen generalized this result further \cite{PALEY2}. After two years, Casazza and Christensen further extended the result \cite{PALEY3}. As long as OVFs are concerned, Sun derived stability results \cite{SUN2}. We now obtain perturbation results for factorable weak OVFs. For this, we need a theorem. This result extends a  result of  Hilding \cite{HILDING}.

 \begin{theorem}\cite{CASAZZAKALTON, PALEY3}\label{cc1}
	Let $ \mathcal{X}, \mathcal{Y}$ be Banach spaces, $ U : \mathcal{X}\rightarrow \mathcal{Y}$ be a bounded invertible operator. If  a bounded linear  operator $ V : \mathcal{X}\rightarrow \mathcal{Y}$ is  such that there exist  $ \alpha, \beta \in \left [0, 1  \right )$ with 
	$$ \|Ux-Vx\|\leq\alpha\|Ux\|+\beta\|Vx\|,\quad \forall x \in  \mathcal{X},$$
	then $ V $ is bounded invertible and 
	$$ \frac{1-\alpha}{1+\beta}\|Ux\|\leq\|Vx\|\leq\frac{1+\alpha}{1-\beta} \|Ux\|, \quad\forall x \in  \mathcal{X};$$
	$$ \frac{1-\beta}{1+\alpha}\frac{1}{\|U\|}\|y\|\leq\|V^{-1}y\|\leq\frac{1+\beta}{1-\alpha} \|U^{-1}\|\|y\|, \quad\forall y \in  \mathcal{Y}.$$
\end{theorem}
\begin{theorem}\label{PERTURBATION RESULT 1}
	Let $  ( \{A_n\}_{n},  \{\Psi_n\}_{n} ) $  be  a factorable weak OVF in $ \mathcal{B}(\mathcal{H}, \mathcal{H}_0)$. Suppose  $\{B_n\}_{n} $ in $ \mathcal{B}(\mathcal{H}, \mathcal{H}_0)$ is such that  there exist $\alpha, \beta, \gamma \geq 0  $ with $ \max\{\alpha+\gamma\|\theta_\Psi (S_{A,\Psi}^*)^{-1}\|, \beta\}<1$ and for all $m=1,2, \dots, $
	\begin{equation}\label{p3}
	\left\|\sum\limits_{n=1}^m(A_n^*-B_n^*)L_n^*y\right\|\leq \alpha\left\|\sum\limits_{n=1}^mA_n^*L_n^*y\right\|+\beta\left\|\sum\limits_{n=1}^mB_n^*L_n^*y\right\|+\gamma \left(\sum\limits_{n=1}^m\|L_n^*y\|^2\right)^\frac{1}{2},\quad \forall y \in \ell^2(\mathbb{N})\otimes \mathcal{H}_0.
	\end{equation} 
	Then  $  ( \{B_n\}_{n},  \{\Psi_n\}_{n} ) $ is a factorable weak OVF  with bounds $ \frac{1-(\alpha+\gamma\|\theta_\Psi (S_{A,\Psi}^*)^{-1}\|)}{(1+\beta)\|(S_{A,\Psi}^*)^{-1}\|}$ and $\frac{\|\theta_\Psi\|((1+\alpha)\|\theta_A\|+\gamma)}{1-\beta} $.
\end{theorem}
\begin{proof}
	For $m=1,2,\dots, $ and for every $ y$ in $ \ell^2(\mathbb{N})\otimes \mathcal{H}_0$, 
	\begin{align*}
	\left\| \sum\limits_{n=1}^mB_n^*L_n^*y\right\|&\leq \left\| \sum\limits_{n=1}^m(A_n^*-B_n^*)L_n^*y\right\|+\left\| \sum\limits_{n=1}^mA_n^*L_n^*y\right\|\\
	&\leq(1+\alpha)\left\| \sum\limits_{n=1}^mA_n^*L_n^*y\right\|+\beta\left\| \sum\limits_{n=1}^mB_n^*L_n^*y\right\|+\gamma\left( \sum\limits_{n=1}^m\|L_n^*y\|^2\right)^\frac{1}{2}
	\end{align*}
	which implies 
	\begin{equation}\label{p1}
	\left\| \sum\limits_{n=1}^mB_n^*L_n^*y\right\|\leq\frac{1+\alpha}{1-\beta}\left\| \sum\limits_{n=1}^mA_n^*L_n^*y\right\|+\frac{\gamma}{1-\beta}\left( \sum\limits_{n=1}^m\|L_n^*y\|^2\right)^\frac{1}{2}, \quad \forall y \in \ell^2(\mathbb{N})\otimes \mathcal{H}_0.
	\end{equation}
	Since  
	$$ \langle y,y\rangle =\langle (I_{\ell^2(\mathbb{N})}\otimes I_{\mathcal{H}_0})y,y\rangle=\left\langle\sum\limits_{n=1}^\infty L_nL_n^* y,y\right\rangle=\sum\limits_{n=1}^\infty\|L_n^* y\|^2 , \quad \forall y \in \ell^2(\mathbb{N})\otimes \mathcal{H}_0,$$
	 Inequality   (\ref{p1})  shows that $\sum_{n=1}^\infty B_n^*L_n^*y $ exists for all  $   y \in \ell^2(\mathbb{N})\otimes \mathcal{H}_0.$ 
	From the continuity of norm, Inequality (\ref{p1}) gives
	\begin{align}\label{p2}
	\left\| \sum\limits_{n=1}^\infty B_n^*L_n^*y\right\|&\leq\frac{1+\alpha}{1-\beta}\left\| \sum\limits_{n=1}^\infty A_n^*L_n^*y\right\|+\frac{\gamma}{1-\beta}\left( \sum\limits_{n=1}^\infty\|L_n^*y\|^2\right)^\frac{1}{2}\nonumber \\
	&=\frac{1+\alpha}{1-\beta}\left\| \theta_A^*y\right\|+\frac{\gamma}{1-\beta}\|y\| , \quad \forall y \in \ell^2(\mathbb{N})\otimes \mathcal{H}_0
	\end{align}
	and this gives $ \sum_{n=1}^\infty B_n^*L_n^* $   is bounded; therefore its adjoint exists, which is $ \theta_B$; Inequality (\ref{p2}) now produces $\|\theta_B^*y\|\leq \frac{1+\alpha}{1-\beta}\left\| \theta_A^*y\right\|+\frac{\gamma}{1-\beta}\|y\| , \forall y \in \ell^2(\mathbb{N})\otimes \mathcal{H}_0 $ and from this $\|\theta_B\|=\|\theta_B^*\|\leq \frac{1+\alpha}{1-\beta}\left\| \theta_A^*\right\|+\frac{\gamma}{1-\beta} =\frac{1+\alpha}{1-\beta}\left\| \theta_A\right\|+\frac{\gamma}{1-\beta}.$
 All in all,  we derived $ S_{B, \Psi}$ is a  bounded linear operator. 
	Continuity of the norm, existence of frame  operators together with Inequality (\ref{p3}) give
	$$ \|\theta_A^*y-\theta_B^*y\|\leq \alpha\|\theta_A^*y\|+\beta\|\theta_B^*y\|+\gamma\|y\|, \quad \forall y \in \ell^2(\mathbb{N})\otimes \mathcal{H}_0$$

	which implies
	$$  \|\theta_A^*(\theta_\Psi (S_{A,\Psi}^*)^{-1} h)-\theta_B^*(\theta_\Psi S_{A,\Psi}^{-1}h)\|\leq \alpha\|\theta_A^*(\theta_\Psi (S_{A,\Psi}^*)^{-1} h)\|+\beta\|\theta_B^*(\theta_\Psi S_{A,\Psi}^{-1} h)\|+\gamma\|\theta_\Psi (S_{A,\Psi}^*)^{-1} h\|, \quad \forall h \in  \mathcal{H}.$$
	But $ \theta_A^*\theta_\Psi (S_{A,\Psi}^*)^{-1}=I_\mathcal{H}$ and $\theta_B^*\theta_\Psi (S_{A,\Psi}^*)^{-1}= S_{B,\Psi} (S_{A,\Psi}^*)^{-1}.$ Therefore 
	\begin{align*}
	\| h- S_{B,\Psi}(S_{A,\Psi}^*)^{-1}h\|
	&\leq \alpha\| h\|+\beta\|S_{B,\Psi} (S_{A,\Psi}^*)^{-1} h\|+\gamma\|\theta_\Psi (S_{A,\Psi}^*)^{-1} h\|\\
	&\leq(\alpha+\gamma\|\theta_\Psi (S_{A,\Psi}^*)^{-1}\|)\|h\|+\beta\|S_{B,\Psi} (S_{A,\Psi}^*)^{-1} h\|, \quad \forall h \in  \mathcal{H}.
	\end{align*}
	Since $ \max\{\alpha+\gamma\|\theta_\Psi (S_{A,\Psi}^*)^{-1}\|, \beta\}<1$, Theorem \ref{cc1} tells that  $S_{B,\Psi} (S_{A,\Psi}^*)^{-1} $ is invertible and 
	$\|(S_{B,\Psi} (S_{A,\Psi}^*)^{-1})^{-1}\| \leq \frac{1+\beta}{1-(\alpha+\gamma\|\theta_\Psi (S_{A,\Psi}^*)^{-1}\|)}.$ From these, we get $(S_{B,\Psi} (S_{A,\Psi}^*)^{-1})S_{A,\Psi}^*=S_{B,\Psi} $ is invertible and 
	\begin{align*}
	\| S_{B,\Psi}^{-1}\|\leq\|(S_{A,\Psi}^*)^{-1}\|\| S_{A,\Psi}^*S_{B,\Psi}^{-1}\| \leq \frac{\|(S_{A,\Psi}^*)^{-1}\|(1+\beta)}{1-(\alpha+\gamma\|\theta_\Psi (S_{A,\Psi}^*)^{-1}\|)}.
	\end{align*}
	Therefore $  ( \{B_n\}_{n},  \{\Psi_n\}_{n} ) $ is a factorable weak  OVF.  Observing that 
	\begin{align*}
	\|S_{B,\Psi}\|\leq \|\theta_\Psi\|\|\theta_B\|\leq \frac{\|\theta_\Psi\|((1+\alpha)\|\theta_A\|+\gamma)}{1-\beta}
	\end{align*}
	 and  $ \|S_{B,\Psi}^{-1}\|^{-1}$ and $\|S_{B,\Psi}\| $ are optimal lower and upper frame bounds for $  ( \{B_n\}_{n},  \{\Psi_n\}_{n} ) $,  we get the frame bounds stated in the theorem.
\end{proof}
\begin{corollary}
	Let $  ( \{A_n\}_{n},  \{\Psi_n\}_{n} ) $  be  a factorable weak OVF  in $ \mathcal{B}(\mathcal{H}, \mathcal{H}_0)$. Suppose  $\{B_n\}_{n} $ in $ \mathcal{B}(\mathcal{H}, \mathcal{H}_0)$ is such that 
	$$ r \coloneqq \sum_{n=1}^\infty\|A_n-B_n\|^2 <\frac{1}{\|\theta_\Psi (S_{A,\Psi}^*)^{-1}\|^2}.$$
	Then $  ( \{B_n\}_{n},  \{\Psi_n\}_{n} ) $ is   a factorable weak OVF  with bounds $ \frac{1-\sqrt{r}\|\theta_\Psi (S_{A,\Psi}^*)^{-1}\|}{\|(S_{A,\Psi}^*)^{-1}\|}$ and ${\|\theta_\Psi\|(\|\theta_A\|+\sqrt{r})} $.
\end{corollary}
\begin{proof}
We try to apply Theorem \ref{PERTURBATION RESULT 1}.	Take $ \alpha =0, \beta=0, \gamma=\sqrt{r}$. Then $ \max\{\alpha+\gamma\|\theta_\Psi (S_{A,\Psi}^*)^{-1}\|, \beta\}<1$ and  for all $m=1,2, \dots, $
	$$ \left\|\sum\limits_{n=1}^m(A_n^*-B_n^*)L_n^*y\right\|\leq \left(\sum\limits_{n=1}^m\|A_n^*-B_n^*\|^2 \right)^\frac{1}{2}\left(\sum\limits_{n=1}^m\|L_n^*y\|^2\right)^\frac{1}{2}\leq \gamma\left(\sum\limits_{n=1}^m\|L_n^*y\|^2\right)^\frac{1}{2}, ~\forall y \in \ell^2(\mathbb{N})\otimes \mathcal{H}_0.$$
\end{proof}
\begin{theorem}\label{OVFQUADRATICPERTURBATION}
	Let $  ( \{A_n\}_{n},  \{\Psi_n\}_{n} ) $  be a factorable weak OVF  in $ \mathcal{B}(\mathcal{H}, \mathcal{H}_0)$. Suppose  $\{B_n\}_{n} $ in $ \mathcal{B}(\mathcal{H}, \mathcal{H}_0)$ is such that  $   \sum_{n=1}^\infty\|A_n-B_n\|^2$ converges, and 
	$\sum_{n=1}^\infty\|A_n-B_n\|\|\Psi_n(S_{A,\Psi}^*)^{-1}\|<1.$
	Then  $  ( \{B_n\}_{n},  \{\Psi_n\}_{n} ) $ is a factorable weak OVF  with bounds 
	\begin{align*}
	\frac{1-\sum_{n=1}^\infty\|A_n-B_n\|\|\Psi_n(S_{A,\Psi}^*)^{-1}\|}{\|(S_{A,\Psi}^*)^{-1}\|}\quad \text{ and } \quad \|\theta_\Psi\|\left(\left(\sum_{n=1}^\infty\|A_n-B_n\|^2\right)^{1/2}+\|\theta_A\|\right) .
	\end{align*}  
\end{theorem}
\begin{proof}
	Let $ \alpha =\sum_{n=1}^\infty\|A_n-B_n\|^2 $ and $\beta =\sum_{n=1}^\infty\|A_n-B_n\|\|\Psi_n(S_{A,\Psi}^*)^{-1}\|$. For  $m=1,2,\dots $ and for every $ y$ in $ \ell^2(\mathbb{N})\otimes \mathcal{H}_0$, 
	
	\begin{align*}
	\left\| \sum\limits_{n=1}^mB_n^*L_n^*y\right\|&\leq \left\| \sum\limits_{n=1}^m(A_n^*-B_n^*)L_n^*y\right\|+\left\| \sum\limits_{n=1}^mA_n^*L_n^*y\right\|\leq \sum\limits_{n=1}^m\|A_n-B_n\|\|L_n^*y\|+\left\| \sum\limits_{n=1}^mA_n^*L_n^*y\right\| \\
	&\leq \left( \sum\limits_{n=1}^m\|A_n-B_n\|^2\right)^\frac{1}{2}\left( \sum\limits_{n=1}^m\|L_n^*y\|^2\right)^\frac{1}{2}+\left\| \sum\limits_{n=1}^mA_n^*L_n^*y\right\|\\
	&\leq \alpha^\frac{1}{2} \left( \sum\limits_{n=1}^m\|L_n^*y\|^2\right)^\frac{1}{2}+\left\| \sum\limits_{n=1}^mA_n^*L_n^*y\right\|=\alpha^\frac{1}{2} \left\langle  \sum\limits_{n=1}^mL_nL_n^*y, y\right\rangle ^\frac{1}{2}+\left\| \sum\limits_{n=1}^mA_n^*L_n^*y\right\|,
	\end{align*}
	which converges to $\sqrt{\alpha}\|y\|+\|\theta_A^*y\|$. Hence 
	$\theta_B$ exists and $\|\theta_B\|\leq \sqrt{\alpha}+\|\theta_A\|$. Therefore  $S_{B,\Psi}=\theta_\Psi^*\theta_B=\sum_{n=1}^\infty\Psi^*_nB_n$ exists.
	Now 
	\begin{align*}
	\|I_\mathcal{H}-S_{B,\Psi}(S_{A,\Psi}^*)^{-1}\|&=\left\|\sum_{n=1}^\infty A_n^*\Psi_n (S_{A,\Psi}^*)^{-1}-\sum_{n=1}^\infty B_n^*\Psi_n (S_{A,\Psi}^*)^{-1}\right\|=\left\|\sum_{n=1}^\infty(A_n^*-B_n^*)\Psi_n (S_{A,\Psi}^*)^{-1}\right\|\\
	&\leq \sum_{n=1}^\infty\|A_n-B_n\|\|\Psi_n (S_{A,\Psi}^*)^{-1}\| =\beta<1.
	\end{align*}
	Therefore $S_{B,\Psi}(S_{A,\Psi}^*)^{-1}$ is invertible and $ \|(S_{B,\Psi}(S_{A,\Psi}^*)^{-1})^{-1}\|\leq 1/(1-\beta)$. Conclusion of frame bounds is similar to proof of Theorem \ref{PERTURBATION RESULT 1}.
\end{proof}

  \section{Acknowledgements}
  We thank  Prof. Victor Kaftal, University of Cincinnati, Ohio for   giving reasons of some of the arguments in the  paper ``Operator-valued frames"  \cite{KAFTAL} coauthored by him. The first author thanks the National Institute of Technology Karnataka (NITK), Surathkal for giving financial support and the present work of the second author was partially supported by Science and Engineering Research Council (SERC), DST, Government of India, through the Fast Track  Scheme for Young Scientists (D.O. No. SR/FTP/MS-050/2011).

 \bibliographystyle{plain}
 \bibliography{reference.bib}

\begin{thebibliography}{10}

\bibitem{ALDROUBI}
Akram Aldroubi.
\newblock Portraits of frames.
\newblock {\em Proc. Amer. Math. Soc.}, 123(6):1661--1668, 1995.

\bibitem{OUTER}
Akram Aldroubi, Carlos Cabrelli, and Ursula~M. Molter.
\newblock Wavelets on irregular grids with arbitrary dilation matrices and
  frame atoms for {$L^2({\Bbb R}^d)$}.
\newblock {\em Appl. Comput. Harmon. Anal.}, 17(2):119--140, 2004.

\bibitem{ARSOVE}
Maynard~G. Arsove.
\newblock The {P}aley-{W}iener theorem in metric linear spaces.
\newblock {\em Pacific J. Math.}, 10:365--379, 1960.

\bibitem{RADU}
Radu Balan.
\newblock Equivalence relations and distances between {H}ilbert frames.
\newblock {\em Proc. Amer. Math. Soc.}, 127(8):2353--2366, 1999.

\bibitem{BALANTHESIS}
Radu~Victor Balan.
\newblock {\em A study of {W}eyl-{H}eisenberg and wavelet frames}.
\newblock ProQuest LLC, Ann Arbor, MI, 1998.
\newblock Thesis (Ph.D.)--Princeton University.

\bibitem{CASAZZAKALTON}
Peter~G. Casazza and Nigel~J. Kalton.
\newblock Generalizing the {P}aley-{W}iener perturbation theory for {B}anach
  spaces.
\newblock {\em Proc. Amer. Math. Soc.}, 127(2):519--527, 1999.

\bibitem{CASAZZASUBSPACE}
Peter~G. Casazza and Gitta Kutyniok.
\newblock Frames of subspaces.
\newblock In {\em Wavelets, frames and operator theory}, volume 345 of {\em
  Contemp. Math.}, pages 87--113. Amer. Math. Soc., Providence, RI, 2004.

\bibitem{CASAZZAFUSION}
Peter~G. Casazza, Gitta Kutyniok, and Shidong Li.
\newblock Fusion frames and distributed processing.
\newblock {\em Appl. Comput. Harmon. Anal.}, 25(1):114--132, 2008.

\bibitem{PALEY3}
Peter~G. Cazassa and Ole Christensen.
\newblock Perturbation of operators and applications to frame theory.
\newblock {\em J. Fourier Anal. Appl.}, 3(5):543--557, 1997.

\bibitem{CHRISTENSENOBLIQUE}
O.~Christensen and Y.~C. Eldar.
\newblock Oblique dual frames and shift-invariant spaces.
\newblock {\em Appl. Comput. Harmon. Anal.}, 17(1):48--68, 2004.

\bibitem{PALEY1}
Ole Christensen.
\newblock Frame perturbations.
\newblock {\em Proc. Amer. Math. Soc.}, 123(4):1217--1220, 1995.

\bibitem{PALEY2}
Ole Christensen.
\newblock A {P}aley-{W}iener theorem for frames.
\newblock {\em Proc. Amer. Math. Soc.}, 123(7):2199--2201, 1995.

\bibitem{CHRISTENSENBOOK}
Ole Christensen.
\newblock {\em An introduction to frames and {R}iesz bases}.
\newblock Applied and Numerical Harmonic Analysis. Birkh\"{a}user/Springer,
  [Cham], second edition, 2016.

\bibitem{CONWAY}
John~B. Conway.
\newblock {\em A course in operator theory}, volume~21 of {\em Graduate Studies
  in Mathematics}.
\newblock American Mathematical Society, Providence, RI, 2000.

\bibitem{CZAJA}
Wojciech Czaja.
\newblock Remarks on {N}aimark's duality.
\newblock {\em Proc. Amer. Math. Soc.}, 136(3):867--871, 2008.

\bibitem{MEYER1}
Ingrid Daubechies, A.~Grossmann, and Y.~Meyer.
\newblock Painless nonorthogonal expansions.
\newblock {\em J. Math. Phys.}, 27(5):1271--1283, 1986.

\bibitem{DUFFIN}
R.~J. Duffin and A.~C. Schaeffer.
\newblock A class of nonharmonic {F}ourier series.
\newblock {\em Trans. Amer. Math. Soc.}, 72:341--366, 1952.

\bibitem{FORNASIERQUASI}
Massimo Fornasier.
\newblock Quasi-orthogonal decompositions of structured frames.
\newblock {\em J. Math. Anal. Appl.}, 289(1):180--199, 2004.

\bibitem{GABARDO}
Jean-Pierre Gabardo and Deguang Han.
\newblock Subspace {W}eyl-{H}eisenberg frames.
\newblock {\em J. Fourier Anal. Appl.}, 7(4):419--433, 2001.

\bibitem{GABARDOHANGROUPLIKE}
Jean-Pierre Gabardo and Deguang Han.
\newblock Frame representations for group-like unitary operator systems.
\newblock {\em J. Operator Theory}, 49(2):223--244, 2003.

\bibitem{KORNELSON}
Deguang Han, Keri Kornelson, David Larson, and Eric Weber.
\newblock {\em Frames for undergraduates}, volume~40 of {\em Student
  Mathematical Library}.
\newblock American Mathematical Society, Providence, RI, 2007.

\bibitem{HANLARSON}
Deguang Han and David~R. Larson.
\newblock Frames, bases and group representations.
\newblock {\em Mem. Amer. Math. Soc.}, 147(697):x+94, 2000.

\bibitem{HANLIMENGTANG}
DeGuang Han, PengTong Li, Bin Meng, and WaiShing Tang.
\newblock Operator valued frames and structured quantum channels.
\newblock {\em Sci. China Math.}, 54(11):2361--2372, 2011.

\bibitem{HEILBOOK}
Christopher Heil.
\newblock {\em A basis theory primer}.
\newblock Applied and Numerical Harmonic Analysis. Birkh\"{a}user/Springer, New
  York, expanded edition, 2011.

\bibitem{HILDING}
Sven~H. Hilding.
\newblock Note on completeness theorems of {P}aley-{W}iener type.
\newblock {\em Ann. of Math. (2)}, 49:953--955, 1948.

\bibitem{HOLUB}
James~R. Holub.
\newblock Pre-frame operators, {B}esselian frames, and near-{R}iesz bases in
  {H}ilbert spaces.
\newblock {\em Proc. Amer. Math. Soc.}, 122(3):779--785, 1994.

\bibitem{KAFTAL}
Victor Kaftal, David~R. Larson, and Shuang Zhang.
\newblock Operator-valued frames.
\newblock {\em Trans. Amer. Math. Soc.}, 361(12):6349--6385, 2009.

\bibitem{KASHINKULIKOVA}
B.~S. Kashin and T.~Yu. Kulikova.
\newblock A remark on the description of frames of general form.
\newblock {\em Mat. Zametki}, 72(6):941--945, 2002.

\bibitem{LI}
Shidong Li.
\newblock On general frame decompositions.
\newblock {\em Numer. Funct. Anal. Optim.}, 16(9-10):1181--1191, 1995.

\bibitem{LIPSEUDO}
Shidong Li and Hidemitsu Ogawa.
\newblock Pseudoframes for subspaces with applications.
\newblock {\em J. Fourier Anal. Appl.}, 10(4):409--431, 2004.

\bibitem{SUN1}
Wenchang Sun.
\newblock {$G$}-frames and {$g$}-{R}iesz bases.
\newblock {\em J. Math. Anal. Appl.}, 322(1):437--452, 2006.

\bibitem{SUN2}
Wenchang Sun.
\newblock Stability of {$g$}-frames.
\newblock {\em J. Math. Anal. Appl.}, 326(2):858--868, 2007.

\bibitem{YOUNG}
Robert~M. Young.
\newblock {\em An introduction to nonharmonic {F}ourier series}, volume~93 of
  {\em Pure and Applied Mathematics}.
\newblock Academic Press, Inc. [Harcourt Brace Jovanovich, Publishers], New
  York-London, 1980.

\end{thebibliography}

\end{document}